\documentclass[10pt, reqno]{amsart}
\usepackage{graphicx}
\usepackage{amsfonts, amssymb} 
\usepackage{mathrsfs,latexsym,mathtools}
\usepackage{tikz-cd}
\usepackage{varwidth}
\usepackage[dvipsnames]{xcolor}
\usepackage{enumerate}
\usepackage{float}
\usepackage[colorlinks=true, citecolor=blue, linkcolor=blue]{hyperref}
\usepackage{parskip} 
\usepackage{fullpage}

\newcommand{\R}{{\mathbb R}}
\newcommand{\N}{{\mathbb N}}
\newcommand{\Z}{{\mathbb Z}}

\newcommand{\C}{{\mathbb C}}
\newcommand{\p}{{\mathbb{P}}}
\newcommand{\B}{{\mathbb{B}}}

\newcommand{\lt}{\left}
\newcommand{\rt}{\right}

\newcommand{\wt}{\widetilde}
\newcommand{\eps}{\varepsilon}
\newcommand{\ph}{\varphi}

\newcommand{\mb}[1]{\mathbf{#1}}
\newcommand{\bs}[1]{\boldsymbol{#1}}
\newcommand{\Sf}[1]{\mathsf{#1}}
\newcommand{\dd}{\textup{d}}

\newcommand{\eqind}{\stackrel{d}{=}}

\newtheorem{theorem}{Theorem}
\newtheorem{app.theorem}{Theorem}[section]
\newtheorem{lemma}{Lemma}
\newtheorem{app.lemma}{Lemma}[section]
\newtheorem{prop}{Proposition}
\newtheorem{app.prop}{Proposition}[section]

\newtheorem{app.cor}{Corollary}[section]
\newtheorem{definition}{Definition}
\newtheorem{app.def}{Definition}[section]
\newtheorem{conjecture}{Conjecture}
\newtheorem{app.conj}{Conjecture}[section]

\DeclareMathOperator{\TW}{TW}

\title{Conditional line-to-point exponential directed last passage percolation under a one-point upper large deviation event}

\author{Rishi Dadlani}
\address{Department of Mathematics, University of Pennsylvania, Philadelphia, PA, 19104, USA}
\email{rdadlani@sas.upenn.edu}

\author{Shaotian Sun}
\address{Department of Mathematics, University of Michigan, Ann Arbor, MI, 48109, USA}
\email{shaotian@umich.edu}

\date{\today}

\begin{document}
\begin{abstract}
Exponential directed last passage percolation is a random maximization model connected to tandem queues and interacting particle systems. We study the line-to-point model of the percolation, in which the starting point is allowed to vary inside a line boundary, conditioned on an atypically large last-passage time to a fixed endpoint.  Extending recent results for the point-to-point model where the starting point is fixed, we develop heuristic predictions for the conditional limit shape and prove corresponding fluctuation results in certain regimes through steepest-descent analysis of exact multi-point formulas. Our calculations suggest that the conditional fluctuations near diagonal are described by Brownian bridges coupled through an additional shared Gaussian component.

\end{abstract}
\maketitle
\section{Introduction and main results}
\label{sec:intro}

\subsection{Model definition}

Let us denote the set of natural numbers as $\N$  and set $\N_0=\N\cup \{0\}$. 
For two points $\mb p= ( p_1,  p_2)$  and $\mb q= (q_1,  q_2)$ in $\N_0^2$ such that $p_1\le q_1$ and $p_2\le  q_2$, 
an up-right path from $\mb p$ to $\mb q$  is a sequence 
$\pi= \{ \mb v_i \}_{i=1}^r$ with $r=q_1+q_2-p_1-p_2+1$, where 
$\mb v_1=\mb p$, $\mb v_r= \mb q$, and $\mb v_{i+1}-\mb v_{i}\in \{ (1,0), (0,1)\}$ for every $i$.

\begin{definition}[Point-to-point Exponential LPP]\label{def:explpp}
Let $\{ \omega_{\mb v} : \mb v \in \N^2 \}$ be a collection of i.i.d.\@ exponential random variables of mean $1$. 
The last passage time from $\mb p$ to $\mb q$ is 
\begin{equation*}
    \mathcal L_{\mb p}(\mb q) = \max_{\pi: \mb p \to \mb q} E(\pi), 
    \qquad E(\pi)= \sum_{\mb v\in \pi} \omega_{\mb v}, 
\end{equation*}
where we take maximum over all up-right paths from $\mb p$ to $\mb q$. 
In particular, when $\mb p=(1,1)$, we write $\mathcal L_{(1,1)}(\mb q)= \mathcal L(\mb q)$. 
We call the 2-dimensional random field $\mathcal L= \{\mathcal L(\mb q) : \mb q \in \N_0^2\}$ the point-to-point exponential last passage percolation.  
For $(\alpha, \beta)\in \R_+^2$, we define
\begin{equation}
    \label{eq:LPPforreal}
    \mathcal L(\alpha, \beta)  = \mathcal L(\lceil\alpha\rceil, \lceil\beta\rceil)
\end{equation} 
where $\lceil\alpha \rceil$ denotes the smallest integer that is larger or equal to $\alpha$. 
\end{definition}

Next, we introduce additional notations for line-to-point exponential LPP. For $r \in \mathbb Z$, let $\mathbb L_r$ denote the line $\{(x,y)\in \mathbb Z^2: x+y=2r\}$.
To avoid notational overhead, we drop the subscript $r$ when $r = 0$. For $\mb v=(v_1,v_2) \in \mathbb Z^2$, we call $\mb v \succ \mathbb L_r$ if $v_1+v_2>2r$. In particular, we say $\mb v \succ \mathbb L$ if $v_1+v_2>0$.

\begin{definition}[Line-to-point Exponential LPP]\label{def:flat-explpp}
Let $\{ \omega_{\mb v} : \mb v \succ \mathbb L_r\}$ be a collection of i.i.d.\@ exponential random variables of mean $1$. Set $\omega_{\mb v}=0$ if $\mb v \in \mathbb L_r$. 
The last passage time from $\mathbb L_r$ to $\mb q$ is 
\begin{equation*}
    \mathcal L_{r}^{\mathrm{flat}}(\mb q) = \max_{\pi: \mb p \to \mb q, \mb p \in \mathbb L_r} E(\pi), 
    \qquad E(\pi)= \sum_{\mb v\in \pi} \omega_{\mb v}, 
\end{equation*}
where we take maximum over all up-right paths from any point $\mb p$ in $\mathbb L_r$ to $\mb q$. 
In particular, when $r=0$, we write $\mathcal L_{0}^{\mathrm{flat}}(\mb q)= \mathcal L^{\mathrm{flat}}(\mb q)$. 
Similarly, we call the 2-dimensional random field $\mathcal L^{\mathrm{flat}}= \{\mathcal L^{\mathrm{flat}}(\mb q) : \mb q \succ \mathbb L\}$ the line-to-point exponential last passage percolation, or flat exponential last passage percolation.
For $(\alpha, \beta)\in \R^2$ satisfying $\alpha + \beta >0$, we define
\begin{equation}
    \label{eq:flat-LPPforreal}
    \mathcal L^{\mathrm{flat}}(\alpha, \beta)  = \mathcal L^{\mathrm{flat}}(\lceil\alpha\rceil, \lceil\beta\rceil)
\end{equation} 
where $\lceil\alpha \rceil$ denotes the smallest integer that is larger or equal to $\alpha$. 
\end{definition}

See Figure~\ref{fig:lpp-geometries} for an illustration of the two models and their different geometries.
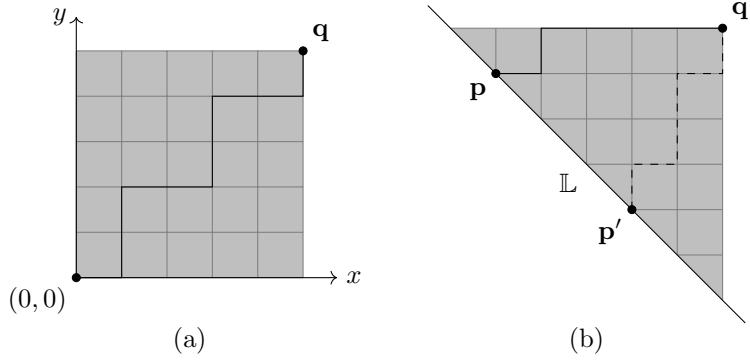
\begin{figure}[h]
\centering
\begin{tikzpicture}[scale=1.5]
% Point-to-point geometry
\begin{scope}
    \fill[lightgray] (0,0) rectangle (2,2);
    \draw[thin,gray,step=0.4] (0,0) grid (2,2);
    \draw[->,thin] (0,0) -- (2.3,0) node[right] {$x$};
    \draw[->,thin] (0,0) -- (0,2.3) node[left] {$y$};
    \draw[thin] (0,0) -- (0.4,0) -- (0.4,0.8)
        -- (1.2,0.8) -- (1.2,1.6) -- (2,1.6) -- (2,2);
    \filldraw (0,0) circle (1pt);
    \filldraw (2,2) circle (1pt);
    \node[below left] at (0,0) {$(0,0)$};
    \node[above right] at (2,2) {$\mathbf q$};
    \node at (1,-0.55) {(a)};
\end{scope}
% Line-to-point geometry (r=0)
\begin{scope}[shift={(4.5,1)}]
    \fill[lightgray] (-1.2,1.2) -- (1.2,-1.2)
        -- (1.2,1.2) -- cycle;
    \foreach \k in {-3,...,3} {
        \draw[thin,gray] ({0.4*\k},{-0.4*\k}) -- ({0.4*\k},1.2);
        \draw[thin,gray] ({-0.4*\k},{0.4*\k}) -- (1.2,{0.4*\k});
    }
    \draw[thin] (-1.4,1.4) -- (1.4,-1.4);
    \node[below left] at (0,0) {$\mathbb L$};
    \draw[thin] (-0.8,0.8) -- (-0.4,0.8) -- (-0.4,1.2) -- (1.2,1.2);
    \draw[thin,dashed] (0.4,-0.4) -- (0.4,0) -- (0.8,0)
        -- (0.8,0.8) -- (1.2,0.8) -- (1.2,1.2);
    \filldraw (-0.8,0.8) circle (1pt);
    \filldraw (0.4,-0.4) circle (1pt);
    \filldraw (1.2,1.2) circle (1pt);
    \node[below left] at (-0.8,0.8) {$\mathbf p$};
    \node[below left] at (0.4,-0.4) {$\mathbf p'$};
    \node[above right] at (1.2,1.2) {$\mathbf q$};
    \node at (0,-1.55) {(b)};
\end{scope}
\end{tikzpicture}
\caption{(a) Point-to-point LPP from $(0,0)$ to $\mathbf q$.
(b) Line-to-point LPP from $\mathbb L$ to $\mathbf q$.
The gray regions contain the admissible paths. The black solid and dashed
lines illustrate up-right paths.}
\label{fig:lpp-geometries}
\end{figure}

\subsection{Background and related results}

The point-to-point exponential LPP is equivalent to some fundamental models in probability and statistical physics,
such as the corner growth model with narrow wedge initial condition, and the continuous-time totally asymmetric
simple exclusion process (TASEP) with step initial condition. Similarly, the line-to-point exponential LPP is equivalent to the corner growth model with flat initial condition, and the continuous-time TASEP with flat initial condition. 

Many unconditional results regarding law of large number and fluctuations have been shown for both point-to-point exponential LPP and line-to-point exponential LPP:

\begin{itemize}
    \item Rost showed the law of large number for point-to-point case in \cite{Rost1981}:
    For every $(x,y)\in \R_+^2$, 
\begin{equation} \label{eq:LLN}
    \lim_{N\to \infty} \frac{\mathcal{L}(xN, yN)}{N} = 
    \bar{\mathcal L}(x,y):=(\sqrt{x} + \sqrt{y})^2
\end{equation}
almost surely and also in probability. 
We also use the convention that $\bar{\mathcal L}_{(x',y')}(x,y)=(\sqrt{x-x'}+\sqrt{y-y'})^2$ so that $\bar{\mathcal L}(x,y)= \bar{\mathcal L}_{(1,1)}(x,y)$.

\item As a direct consequence of Johansson's point-to-point shape theorem \cite{Joh00}, we have the law of large number for line-to-point case: 
\begin{align}
\label{eq:uncondltp-lln}
    \lim_{N \to \infty}\frac{\mathcal L^{\mathrm{flat}}(xN,yN)}{N}
    &=
    \bar{\mathcal L}^{\mathrm{flat}}(x,y) \notag\\
    &:=
    \sup_{-y\leq k\leq x}
    \left(\sqrt{x-k}+\sqrt{y+k}\right)^2
    =
    2(x+y)
\end{align}
almost surely and also in probability. Indeed, when we take supremum over $-y\leq k \leq x$, 
the unique maximizer is $k_*=\frac{x-y}{2}$.
Thus, the macroscopic starting point is the orthogonal projection of
$(x,y)$ onto $\mathbb L$.  In particular, for every $a>0$,
\begin{equation}
    \lim_{N \to \infty}\frac{\mathcal L^{\mathrm{flat}}(aN,aN)}{N}=4a
\end{equation}
almost surely and also in probability, and the corresponding macroscopic starting point is the origin.

\item The convergence of the one-point distributions for point-to-point case was obtained in \cite{Joh00} by Johansson: 
 \begin{equation}
    \label{eq:uncond-ptp-fluc}
    \lim_{N\to \infty} \frac{\mathcal L(xN, yN)-  \bar{\mathcal L}(x,y) N}{ (xy)^{-1/6}(\sqrt{x}+\sqrt{y})^{4/3}  N^{1/3}}  \eqind \TW_2
\end{equation} 
for every $(x,y)\in \R_+^2$, where $\TW_2$ is the GUE Tracy-Widom distribution. 
Here, $\eqind$ denotes convergence in distribution.

\item The convergence of the one-point distributions for line-to-point case was stated in \cite{sasamoto_spatial_2005, borodin_fluctuation_2007}: 
\begin{equation}
    \label{eq:uncond-ltp-fluc}
    \lim_{N\to \infty} \frac{\mathcal L^{\mathrm{flat}}(xN, xN)-  \bar{\mathcal L}^{\mathrm{flat}}(x,x) N}{ 2^{2/3}x^{1/3}N^{1/3}}  \eqind \TW_1
\end{equation} 
for every $x>0$, where $\TW_1$ is the GOE Tracy-Widom distribution. 

\item After appropriate scaling, for given $x>0$, the multi-point distributions for line-to-point case converges to the Airy$_1$ process $\mathcal A_1$: 
\begin{equation}
\label{eq:multi-point}
    s \mapsto \frac{{{\mathcal L}^{\mathrm{flat}}(xN+s(2xN)^{2/3},xN-s(2xN)^{2/3})}-\bar{\mathcal L}^{\mathrm{flat}}(x,x)N}{2^{4/3}x^{1/3}N^{1/3}} \Rightarrow 2^{1/3}\mathcal A_{1}(2^{-2/3}s)
\end{equation}
in the sense of convergence of finite-dimensional distributions \cite{borodin_fluctuations_2006, borodin_fluctuation_2007, borodin_large_2008}. It is worth mentioning that \cite{matetski2021} obtained formulae for the finite dimensional distributions of such limiting objects for a general class of initial data. In particular, we have the convergence of multi-point distributions for point-to-point and line-to-point case: 

\begin{enumerate}
    \item For given $(x,y)\in \R^2_+$, 
\begin{equation}
   \label{eq:uncond-ptp-field-fd-convergence}
    \lim_{N \to \infty} \frac{\mathcal L \left(t x N + s \frac{2 x^{2/3} y^{1/6}}{(\sqrt{x}+\sqrt{y})^{1/3}}  N^{2/3}, {t} y N - s \frac{2 x^{1/6} y^{2/3}}{(\sqrt{x}+\sqrt{y})^{1/3}}  N^{2/3}\right) - \bar{\mathcal L}(tx,ty)N}
    {(xy)^{-1/6}(\sqrt{x}+\sqrt{y})^{4/3}N^{1/3}} 
    \stackrel{f.d.d.}{=\!=\!=} \mathsf{H}_{\mathrm{step}}(s,t ) 
\end{equation} 
in the sense of finite-dimensional distributions for $(s,t) \in \R\times \R_+$,
where $\mathsf{H}_{\mathrm{step}}$ denotes the KPZ fixed point with narrow wedge initial condition;

\item For given $x\in \R_+$, 
\begin{equation}
   \label{eq:uncond-ltp-field-fd-convergence}
    \lim_{N \to \infty} \frac{{{\mathcal L}^{\mathrm{flat}}(txN+s(2xN)^{2/3},txN-s(2xN)^{2/3})}-\bar{\mathcal L}^{\mathrm{flat}}(tx,tx)N}{2^{4/3}x^{1/3}N^{1/3}} 
    \stackrel{f.d.d.}{=\!=\!=} \mathsf{H}_{\mathrm{flat}}(s,t ) 
\end{equation} 
in the sense of finite-dimensional distributions for $(s,t) \in \R\times \R_+$,
where $\mathsf{H}_{\mathrm{flat}}$ denotes the KPZ fixed point with flat initial condition.
\end{enumerate}

\item The upper large deviation result for line-to-point case follows from
\cite[Theorem~1.6]{Joh00}, together with a union bound:
\begin{equation}
    \lim_{N \to \infty} \frac{1}{N} \log \mathbb P(\mathcal L^{\mathrm{flat}}(aN, aN) > \ell N) = -\Sf J(\ell), \quad \text{for }\ell>\bar{\mathcal L}^{\mathrm{flat}}(a,a)=4a, %\footnote{This result is dervied as a special case of \cite[Proposition~3.4]{bhatia_moderate_2020}. The proposition is the specific case of \cite[Theorem~10.5]{basu_last_2016} in general line-to-point exponential LPP case.}
\end{equation}
where \begin{equation} \label{eq:large_deviation_fn}
    \Sf J(\ell)= \sqrt{D} + 2a \log \left(\frac{\ell - \sqrt{D}}{\ell + \sqrt{D}}\right), 
\end{equation}
with 
\begin{equation}
\label{eq:D}
    D=\ell^2-4a\ell.
\end{equation}
\end{itemize}

The goal of the paper is to investigate the behavior of line-to-point exponential LPP conditioned on the event that the last passage time at a specific site is larger than expected. The same problem was studied for the point-to-point exponential LPP in \cite{BCT} and \cite{baik_one-point_2026}. Some known results for the point-to-point exponential LPP under the conditioning are: 
\begin{itemize}
    \item Conditional law of large number for point-to-point exponential LPP was heuristically conjectured in \cite[Conjecture~1.3]{BCT}, and was proved partially. 
    \item Conditional diagonal multi-point fluctuations for point-to-point exponential LPP was shown in \cite[Theorem~1.4]{BCT}.
    \item Conditional fluctuations were computed for two-point distribution in some regions in \cite[Theorem~1.5]{BCT}, and for one-point distribution in some other regions in \cite{baik_one-point_2026}. 
    \item Conditional fluctuations in other regions for point-to-point exponential LPP was heuristically conjectured in \cite[Conjecture~1.6]{BCT}.
\end{itemize}
There are also results on conditional fluctuations in the upper-tail regime for the KPZ fixed points, with both narrow wedge initial conditions and flat initial condition in \cite[Remark~1.5]{liu2022}.

\subsection{Main results}
\subsubsection{Conjecture on conditional law of large numbers}

The first result of this paper is the conjecture on conditional law of large number for line-to-point exponential LPP. It may be compared with unconditional line-to-point law of large number \eqref{eq:uncondltp-lln} and conditional point-to-point law of large number \cite[Conjecture~1.3]{BCT}.

Fix $a>0$ and $\ell>4a$.  Following the convention of
\cite[Equation~(1.9)]{BCT}, for every event $E$, we interpret
\begin{equation}\label{eq:flat-conditioning-convention}
    \mathbb P\left(
        E\,\middle|\,
        \mathcal L^{\mathrm{flat}}(aN,aN)=\ell N
    \right)
    :=
    \lim_{\epsilon\downarrow0}
    \frac{
        \mathbb P\left(
            E\cap
            \left\{
                \mathcal L^{\mathrm{flat}}(aN,aN)
                \in(\ell N-\epsilon,\ell N+\epsilon)
            \right\}
        \right)
    }{
        \mathbb P\left(
            \mathcal L^{\mathrm{flat}}(aN,aN)
            \in(\ell N-\epsilon,\ell N+\epsilon)
        \right)
    }.
\end{equation}

To state the conjecture, define
\begin{equation}
    M
    =
    \frac{\ell-2a+\sqrt D}{2a},
    \qquad
    m=M^2
    =
    \frac{\ell-2a+\sqrt D}
         {\ell-2a-\sqrt D}.
    \label{eq:flat-M-and-m}
\end{equation}

Introduce two regions
\begin{align*}
    \mathcal R_1
    &=
    \left\{
        (x,y)\in(a,\infty)^2:
        \frac1m<\frac{y-a}{x-a}<m, \left|\sqrt{x-a}-\sqrt{y-a}\right|
        \leq\sqrt{\ell-4a}
    \right\},
\\
    \mathcal R_2
    &=
    \left\{
        (x,y)\in(0,\infty)^2:
        \frac1{M}
        \leq\frac yx\leq
        M
    \right\}\setminus\mathcal R_1.
\end{align*} See Figure $\ref{fig:Regions}$ for the picture
of the regions $\mathcal R_1$ and $\mathcal R_2$.

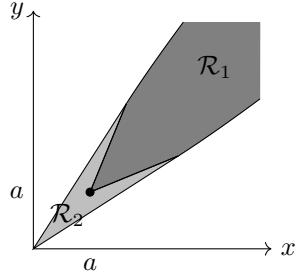
\begin{figure}[h]
\centering
\begin{tikzpicture}[scale=1.5]
    \pgfmathsetmacro{\c}{0.317506}
    \pgfmathsetmacro{\yright}{2 - 2*\c*sqrt(2-0.5) + \c*\c}
    \pgfmathsetmacro{\xtop}{\yright}

    % R2: outer light-gray region
    \fill[lightgray]
        (0,0) -- (1.28,0.82) -- (0.5, 0.5) -- (0.82,1.28) -- cycle;

    % R1: full dark-gray region
    \fill[gray]
        (0.5,0.5)
        -- (1.28,0.82)
        plot[domain=1.28:2,samples=100]
            (\x, {\x - 2*\c*sqrt(\x-0.5) + \c*\c})
        -- (2,2)
        -- (\xtop,2)
        plot[domain=2:1.28,samples=100]
            ({\x - 2*\c*sqrt(\x-0.5) + \c*\c}, \x)
        -- (0.82,1.28)
        -- cycle;
        \fill[gray]
        (0.5,0.5)
        -- (0.82,1.28)
        plot[domain=2:1.28,samples=100]
            ({\x - 2*\c*sqrt(\x-0.5) + \c*\c}, \x)
        -- (2, \xtop)
        -- (2,2)
        plot[domain=1.28:2,samples=100]
            (\x, {\x - 2*\c*sqrt(\x-0.5) + \c*\c})
        -- (1.28,0.82)
        -- cycle;
    \fill[gray]
        (0.5,0.5)
        -- (1.28,0.82)
        -- (1.5, 1.5)
        -- (0.82,1.28);

    % Outer boundary of R2
    \draw[thin] (0.82,1.28) -- (0,0) -- (1.28,0.82);
    \draw[thin] (0.82,1.28) -- (0.5,0.5) -- (1.28,0.82);;

    % Boundary of R1
    \draw[thin] (0.5,0.5) -- (1.28,0.82);
    \draw[thin] (0.5,0.5) -- (0.82,1.28);
    \draw[thin] plot[domain=1.28:2,samples=100]
        (\x, {\x - 2*\c*sqrt(\x-0.5) + \c*\c});
    \draw[thin] plot[domain=2:1.28,samples=100]
        ({\x - 2*\c*sqrt(\x-0.5) + \c*\c}, \x);

    % Axes
    \draw[->, thin] (0,0) -- (2.1,0) node[right] {$x$};
    \draw[->, thin] (0,0) -- (0,2.1) node[left] {$y$};

    % Labels
    \node at (1.6,1.6) {$\mathcal R_1$};
    \node at (0.3,0.3) {$\mathcal R_2$};

    % Mark point a
    \filldraw (0.5,0.5) circle (1pt);
    \node[below] at (0.5,0) {$a$};
    \node[left] at (0,0.5) {$a$};
\end{tikzpicture}
\caption{The dark gray region is $\mathcal R_1$. The light gray region is $\mathcal R_2$.}
\label{fig:Regions}
\end{figure}

\begin{conjecture}[Conditional flat law of large numbers]
\label{conj:lln}
Fix $a>0$ and $\ell>4a$. Conditional on
$\mathcal L^{\mathrm{flat}}(aN,aN)=\ell N$,
where the conditioning is understood in the sense of
\eqref{eq:flat-conditioning-convention}, we have, for every fixed
\(x,y>0\),
\begin{equation*}
    \lim_{N \to \infty}\frac{\mathcal L^{\mathrm{flat}}(xN,yN)}{N}
    =
    h^{\mathrm{flat}}(x,y)
\end{equation*}
in conditional probability, where
\begin{equation*}
    h^{\mathrm{flat}}(x,y)
    =
    \begin{cases}
        \displaystyle
        \ell+
        \bar{\mathcal L}_{(a,a)}(x,y),
        & (x,y)\in\mathcal R_1,\\[8pt]
        \displaystyle
        \frac{\ell}{2a}(x+y)
        -
        \frac{\sqrt D}{2a}|x-y|,
        & (x,y)\in\mathcal R_2,\\[8pt]
        \displaystyle
        \bar{\mathcal L}^{\mathrm{flat}}(x,y),
        & \text{otherwise}.
    \end{cases}
\end{equation*}
Equivalently, for every $\epsilon>0$,
\begin{equation*}
    \lim_{N \to \infty}\mathbb P\left(
        \left|
            \frac{\mathcal L^{\mathrm{flat}}(xN,yN)}{N}
            -
            h^{\mathrm{flat}}(x,y)
        \right|>\epsilon
        \,\middle|\,
        \mathcal L^{\mathrm{flat}}(aN,aN)=\ell N
    \right)
    =0.
\end{equation*}
\end{conjecture}

\subsubsection{Theorem for conditional diagonal multi-point fluctuations}

The second result of the paper is the conditional diagonal multi-point fluctuations in line-to-point exponential LPP. Compare it with unconditional line-to-point multi-point fluctuations \eqref{eq:multi-point} and conditional point-to-point diagonal multi-point fluctuations \cite[Theorem~1.4]{BCT}.
\begin{theorem}[Conditional diagonal multi-point fluctuations in flat exponential LPP]
\label{thm:cond-diag-fluc}
Fix $a > 0$ and $\ell> \bar{{\mathcal L}}^{\mathrm{flat}}(a,a)=4a$. Let $D=\ell^2-4a\ell$. 
Define the positive real number  $\sigma= \sqrt{\frac{\ell}{2a}} D^{1/4}$. 
Then, 
\begin{equation} \begin{split} \label{eq:diagfluc}
    &\mathrm{Law} \left( \left\{ \frac{ {\mathcal L}^{\mathrm{flat}}( t aN +  s \frac{a\sigma}{\sqrt{D}}   N^{1/2}, t aN -  s \frac{a\sigma}{\sqrt{D}}  N^{1/2})- t \ell N}{\sigma N^{1/2}} \right\}_{(s,t)\in \R\times (0,1)}
    \, \bigg| \,  {\mathcal L}^{\mathrm{flat}}(aN, aN)= \ell N \right) \\
    &\xrightarrow{\textit{f.d.d.}}  \mathrm{Law} \left( \{ \B_1(t) - |\B_2(t) -(1-t)Z - s| \}_{(s,t)\in \R\times (0,1)} \right)
\end{split} \end{equation}
as $N\to \infty$, where 
\begin{equation}
\label{eq:defcd}
    \B_1(t)= \frac{ \B_+(t)+  \B_-(t)}{\sqrt{2}},
    \quad 
    \B_2(t)= \frac{ \B_+(t)-  \B_-(t)}{\sqrt{2}} 
    \quad\text{and}\quad 
    Z \sim \mathcal{N}(0,1)
\end{equation}
for two independent standard Brownian bridges $\B_+$ and $\B_-$; thus, $\B_1$ and $\B_2$ are also two independent standard Brownian bridges.
\end{theorem}

\subsubsection{Conjecture on conditional fluctuations}

The last result of the paper is the conjecture on conditional fluctuations in line-to-point exponential LPP. Compare it with unconditional line-to-point multi-point fluctuations \eqref{eq:multi-point} again, and conjecture on conditional point-to-point fluctuations \cite[Conjecture~1.6]{BCT}.

We retain the notation \(D\), \(M\), \(m\),
\(\mathcal R_1\), \(\mathcal R_2\), and \(h^{\mathrm{flat}}\) from
 Conjecture~\ref{conj:lln}, and $\sigma$ from  Theorem~\ref{thm:cond-diag-fluc}. 

Let \(\mathcal R_i^\circ\) denote the interior of \(\mathcal R_i\).
We restrict the conjecture to the interiors of the relevant regions,
since the interface curves correspond to a competition between two
contributions of the same first-order size. On these curves, the
competition is resolved at the level of the second-order fluctuations
and should therefore produce separate transition laws.

\begin{conjecture}[Conditional flat fluctuations]
\label{conj:flat-fluctuations}
Fix \(a>0\) and \(\ell>4a\), and condition on
\[
    \mathcal L^{\mathrm{flat}}(aN,aN)=\ell N
\]
in the sense of \eqref{eq:flat-conditioning-convention}. The following
statements are expected to hold.

\begin{enumerate}
\item[(a)] For every fixed \((x,y)\in\mathcal R_1^\circ\),
\begin{equation}
    \frac{
        \mathcal L^{\mathrm{flat}}(xN,yN)
        -
        h^{\mathrm{flat}}(x,y)N
    }{
        \bigl((x-a)(y-a)\bigr)^{-1/6}
        \bigl(\sqrt{x-a}+\sqrt{y-a}\bigr)^{4/3}
        N^{1/3}
    }
    \xrightarrow{d}
    \TW_2 .
    \label{eq:flat-end-fluctuation-conjecture}
\end{equation}

\item[(b)] For every fixed \((x,y)\in
\mathcal R_2^\circ\cap\{(x,y):x\neq y\}\), define the maximizing attachment time $t_c$ by
\begin{equation}
    t_c=t_c(x,y)
    =
    \begin{cases}
        \displaystyle
        \frac{my-x}{a(m-1)},
        & \quad y<x,\\[8pt]
        \displaystyle
        \frac{mx-y}{a(m-1)},
        & \quad y>x,
    \end{cases}
    \label{eq:flat-off-diagonal-tc}
\end{equation}
then

\begin{equation}
    \frac{
        \mathcal L^{\mathrm{flat}}(xN,yN)
        -
        h^{\mathrm{flat}}(x,y)N
    }{
        \sqrt{2}\,\sigma N^{1/2}
    }
    \xrightarrow{f.d.d.}
    \begin{cases}
        \displaystyle
        \B_+(t_c)
        -
        \frac{1-t_c}{\sqrt{2}}Z,
        & y<x,\\[10pt]
        \displaystyle
        \B_-(t_c)
        +
        \frac{1-t_c}{\sqrt{2}}Z,
        & y>x.
    \end{cases}
     \label{eq:flat-middle-fluctuation-conjecture}
\end{equation}
Here, \(\B_+\), \(\B_-\) are two independent standard Brownian bridges, and \(Z\) is the standard normal random variable. Moreover, they are the same random objects as in
Theorem~\ref{thm:cond-diag-fluc}. The fluctuations on the diagonal portion
of \(\mathcal R_2\) are already described by that theorem.

\item[(c)] For every fixed
$
    (x,y)\in
    \mathbb R_+^2
    \setminus
    \overline{\mathcal R_1\cup\mathcal R_2},
$
\begin{equation}
    \frac{
        \mathcal L^{\mathrm{flat}}(xN,yN)
        -
        h^{\mathrm{flat}}(x,y)N
    }{
        2^{1/3}(x+y)^{1/3}N^{1/3}
    }
    \xrightarrow{d}
    \TW_1 .
    \label{eq:flat-unaffected-fluctuation-conjecture}
\end{equation}
\end{enumerate}
\end{conjecture}

\subsection{Outline of the paper}

The paper is organized as follows. In Section \ref{sec:flat-conditional-lln} and \ref{sec:flat-conditional-fluctuation-heuristic}, we present the heuristic argument for Conjecture~\ref{conj:lln} and Conjecture~\ref{conj:flat-fluctuations}, where the reasoning is similar to that in \cite{BCT} for point-to-point case. In Section \ref{sec:multi-point-formula}, we present explicit
formulas for the multi-time conditional distributions. In Section \ref{sec:diag-fluc}, we give the proof of Theorem~\ref{thm:cond-diag-fluc}. The proof is based on analysis of explicit multi-point distribution formulas for the exponential LPP by Liu in \cite{liu2021}. It is further adapted from the proof in \cite{liu2022} for the KPZ fixed point with flat initial condition.

\subsection*{Acknowledgments}

This work was carried out as part of the 2026 University of Michigan Mathematics REU. 
The authors are grateful to Professor Jinho Baik for his guidance and mentorship 
throughout the project, and to Tian Dong for valuable discussions and assistance.

\section{Conditional Law of Large Numbers Heuristics}
\label{sec:flat-conditional-lln}
Without loss of generality, we can assume that the conditioning point has equal coordinates for the line-to-point model.  Given a conditioning point
$(aN,bN)$, set the change of coordinate $\tau$ as
\begin{equation*}
    \tau:(x,y) \mapsto (x-\frac{a-b}{2},y+\frac{a-b}{2}).
\end{equation*}
Then $\tau(a,b)=(\frac{a+b}{2},\frac{a+b}{2})$,
and $\tau$ preserves the anti-diagonal since 
$\tau(r,-r)=(r-\frac{a-b}{2},-(r-\frac{a-b}{2}))\in\mathbb L.$
Since the exponential environment is translation invariant, the flat problem conditioned at
$(aN,bN)$ is reduced to the problem conditioned at $(\frac{a+b}{2}N,\frac{a+b}{2}N)$.
Thus, it is enough to analyze a diagonal conditioning point
$(aN,aN)$.

We next give a heuristic argument for Conjecture~\ref{conj:lln}.

\subsection{The favorable-path ansatz}

The unconditional flat variational problem has a unique macroscopic
optimizer at the orthogonal projection of the endpoint onto
\(\mathbb L\).  In particular, the maximizing path to
\((aN,aN)\) begins within a submacroscopic distance of the origin.
We therefore make the assumption that the upper-tail conditioning creates
an atypically favorable path running approximately along the
diagonal from $(0,0)$ to $(aN,aN)$.  Assuming that the large passage time is distributed approximately uniformly along this
diagonal line, we expect $\mathcal L^{\mathrm{flat}}(taN,taN)
    \approx t\ell N$ for $t\in [0,1]$.

\subsection{The backbone contribution}

Once this favorable path from the origin is assumed, we may follow the
same heuristic argument as in \cite[Section~2.1]{BCT}.  For a general
endpoint $(xN,yN)$, a path affected by the conditioning first follows
the favorable path up to $(taN,taN)$ and then travels typically from
$(taN,taN)$ to $(xN,yN)$.  Its predicted passage time, divided by $N$,
is
\begin{equation*}
    H(t)
    =
    t\ell+
    \left(\sqrt{x-at}+\sqrt{y-at}\right)^2,
    \qquad
    0\leq t\leq
    \min\left\{\frac xa,\frac ya,1\right\}.
\end{equation*}
Thus the contribution of paths using the favorable path is
\[
    h(x,y)
    :=
    \sup_{0\leq t\leq\min\{x/a,y/a,1\}}
    H(t).
\]
This is precisely the variational problem considered in
\cite[Section~2.1]{BCT}, specialized to a diagonal conditioning point. We introduce the optimizer regions:
\begin{align*}
    \Omega_1
    &=
    \left\{
        (x,y)\in(a,\infty)^2:
        \frac1m<\frac{y-a}{x-a}<m
    \right\},
\\
    \Omega_2
    &=
    \left\{
        (x,y)\in(0,\infty)^2:
        \frac1m<\frac yx<m
    \right\},
\end{align*}
and their optimization for point-to-point case gives
\begin{equation}
    h(x,y)
    =
    \begin{cases}
        \displaystyle
        \ell+
        \left(\sqrt{x-a}+\sqrt{y-a}\right)^2,
        & (x,y)\in\Omega_1,\\[8pt]
        \displaystyle
        \frac{\ell}{2a}(x+y)
        -
        \frac{\sqrt D}{2a}|x-y|,
        & (x,y)\in\Omega_2\setminus\Omega_1,\\[8pt]
        \displaystyle
        \left(\sqrt{x}+\sqrt{y}\right)^2,
        & (x,y)\notin\Omega_2.
    \end{cases}
    \label{eq:flat-optimized-backbone-profile}
\end{equation}

\subsection{Competition with the unconditional flat profile}

The only additional feature in the flat model is that when it does not use the conditioned favorable path, it has a different maximizing path
to $(xN,yN)$.  Instead of starting from the origin, it starts from the unique macroscopic optimizing point at the orthogonal projection of $(xN,yN)$ onto $\mathbb{L}$. Such paths contribute the ordinary flat
shape $\bar{\mathcal L}^{\mathrm{flat}}(x,y)=2(x+y)$.  We therefore
predict
\begin{equation}
    h^{\mathrm{flat}}(x,y)
    =
    \max\left\{
       h(x,y),\,2(x+y)
    \right\}.
    \label{eq:flat-complete-variational-prediction}
\end{equation}
In other words, the backbone contribution is exactly the one obtained
from the heuristic of \cite[Section~2.1]{BCT}; the outer maximum with the
unconditional flat shape is the only difference.

It remains to determine where the backbone contribution dominates the
unconditional flat shape.  Outside $\Omega_2$, the last case of
\eqref{eq:flat-optimized-backbone-profile} satisfies
\[
    \left(\sqrt{x}+\sqrt{y}\right)^2\leq2(x+y),
\]
so the conditioning has no first-order effect there.  In $\Omega_1$,
\begin{align*}
    &\ell+
    \left(\sqrt{x-a}+\sqrt{y-a}\right)^2
    \geq2(x+y) \notag\\
    &\hspace{20mm}\Longleftrightarrow\quad
    \left|\sqrt{x-a}-\sqrt{y-a}\right|
    \leq\sqrt{\ell-4a},
\end{align*}
which is precisely the condition defining $\mathcal R_1$.  In
$\Omega_2\setminus\Omega_1$,
\begin{align*}
    \frac{\ell}{2a}(x+y)
    -
    \frac{\sqrt D}{2a}|x-y|
    \geq2(x+y)
    &\Longleftrightarrow
    \frac{|x-y|}{x+y}
    \leq
    \sqrt{\frac{\ell-4a}{\ell}} \notag\\
    &\Longleftrightarrow
    \frac1M\leq\frac yx\leq M,
\end{align*}
which is the condition defining $\mathcal R_2$.  Substituting these
comparisons into \eqref{eq:flat-complete-variational-prediction} yields
the piecewise function stated in Conjecture~\ref{conj:lln}, which completes the heuristic argument for Conjecture~\ref{conj:lln}.

\section{Conditional Multi-Point Distributions}
\label{sec:multi-point-formula}
We prove Theorem \ref{thm:cond-diag-fluc} by computing the limit of an explicit conditional multi-point distribution. Unless otherwise noted, we use the same formulae, strategy, and notation as described in \cite[Section~3]{BCT}.

We will now repeat the notation from \cite{BCT}, often verbatim, for clarity.
\subsection{Notation}
\label{sec:multipoint-notation}
\phantomsection\label{not:cauchy-determinant}
Let $\mathbf r = (r_1 ,\dots r_n)$ and $\mathbf s = (s_1, \dots, s_n)$ be vectors in $\C^n$. We define \[
K_n(\mathbf r | \mb s) = \det\lt[\frac{1}{r_i - s_j}\rt]_{i,j= 1}^n = \frac{\prod_{1 \leq i < j \leq n} (r_i - r_j) (s_j - s_i)}{\prod_{i,j=1 }(r_i-s_j)}
\]
as the Cauchy determinant of $\mb r, \mb s$.\footnote{The conversion
from Liu's Vandermonde convention to this Cauchy-determinant convention
is given in Appendix~\ref{app:nc-cauchy}.}\phantomsection\label{not:terminal-sum} Define \[
S_n(\mb r | \mb s) = \sum_{i = 1}^n (r_i - s_i).\footnote{The factor $S_n$ arises when differentiating the exponential weights;
see \eqref{eq:app-nc-derivative}.}
\]

For $\mb n  = (n_1, \dots, n_m) \in \N^m$, define, as in \cite{BCT} (3.5), the rational function \begin{equation}
\label{eq:Pi-main}
        \Pi_{\mb n}(\bs \xi, \bs \eta) 
        = K_{n_1}(\bs \eta^{1}|\bs \xi^{1}) \left[ \prod_{i=1}^{m-1} K_{n_i+n_{i+1}}(\bs \xi^{i}, \bs \eta^{i+1}|\bs \eta^{i}, \bs \xi^{i+1}) \right] K_{n_m}(\bs \xi^{m}|\bs \eta^{m}) S_{n_m}(\bs \xi^{m}|\bs \eta^{m}),
\end{equation} 
\phantomsection\label{not:integration-variables}where $\bs \xi = (\bs \xi^1, \dots, \bs \xi^m)$, and $\bs \eta = (\bs \eta^1, \dots, \bs \eta^m)$, with each component $\bs \xi^i, \bs \eta^i \in \C^{n_i}$.\footnote{Their correspondence with Liu's variables is listed in
Appendix~\ref{app:nc-parameters}, and the product of Vandermonde factors
leading to $\Pi_{\mb n}$ is identified in \eqref{eq:app-nc-vandermonde}.}

\phantomsection\label{not:exponential-weight}
For $M,N\in \N$ and $T\in \R_+$, define the function\footnote{Throughout,  $\log$ denotes the branch that is analytic in $\C\setminus i \overline{ \R_-}$ and such that $\log 1=0$.}  \[ 
    f_{M,N,T}(z) := \frac{z^N e^{Tz}}{(z+1)^M} = e^{ N\log z - M\log(z+1) + Tz}.\]
For $\mb M = (M_1,\dots, M_m)\in \N^m$, $\mb N = (N_1,\dots, N_m)\in \N^m$, and  $\mb T = (T_1,\dots, T_m)\in \R_+^m$,  define
\begin{equation}
\label{eq:weight-product-main}
    \Sf F^{(\mb n)}_{\mb M, \mb N, \mb T} (\bs \xi, \bs \eta) 
    =  \prod_{i=1}^m \prod_{k_i=1}^{n_i} \frac{\Sf f_i(\xi_{k_i}^{i})}{\Sf f_i(\eta_{k_i}^{i})}, 
    \qquad \Sf f_i(z)= \frac{f_{M_i, N_i,T_i}(z)}{f_{M_{i-1},N_{i-1},T_{i-1}}(z)}.
\end{equation}
where $\bs \xi$,  $\bs \eta$ are as before.  We set $M_0=N_0=T_0=0$ in the above formula, then $f_{M_0, N_0,T_0}(z) = 1$.\footnote{The functions $f_{M_i,N_i,T_i}$ and $\Sf f_i$ correspond to Liu's
$F_i$ and his side-dependent $f_i$ as explained in
\eqref{eq:app-nc-weights}--\eqref{eq:app-nc-F}.} 

\phantomsection\label{not:multipoint-contours}
Let  
\[C^{\text{in}}_{m,L},\dots,C^{\text{in}}_{2,L}, C_{1,L}, C^{\text{out}}_{2,L}, \dots, C^{\text{out}}_{m,L}
\]
be $2m-1$ small circles, nested inside to outside, enclosing the point $-1$.  Similarly, let 
\[
    C^{\text{in}}_{m,R},\dots, C^{\text{in}}_{2,R}, C_{1,R}, C^{\text{out}}_{2,R}, \dots, C^{\text{out}}_{m,R}
\]
be $2m-1$ small circles, nested inside to outside, enclosing the point $0$ and disjoint from the previous circles. \footnote{Unlike in \cite{BCT}, our formulae often have an additional consideration that will require either the condition $-1 - C_{1,R} = C_{1,L}$, or a similar condition for a modified set of contours. The delta-kernel interpretation is explained in
Appendix~\ref{app:nc-flat-factor}; the conversion from Liu's contour
measures is given in Appendix~\ref{app:nc-measures}.
} All closed contours in this paper are oriented counterclockwise, unless otherwise noted. The orientations for infinite contours will be given as they appear.

We reproduce Figure 8 in \cite{BCT} as Figure \ref{fig:allContours} below. These are the contours for the case when $m=2$.

\begin{figure}[h]
\begin{center}
\begin{tikzpicture}[scale=1.3]
    % Draw x-axis with arrows on both ends
    \draw[->] (-3.5,0) -- (1.5,0); %node[right] {$\mathbb{R}$};
    % Draw y-axis with arrows on both ends
    \draw[->] (0,-1.25) -- (0,1.25); % node[above] {$i\mathbb{R}$};
    % Draw circle at (-1, 0) with radius = 1.5
    \node[below] at (-2,0) {$-1$};  % Label -1 on the x-axis       
    \filldraw[black] (-2,0) circle (1pt);
    \draw[thick] (-2,0) circle [radius=0.45];
    \draw[thick] (-2,0) circle [radius=0.65];
    \draw[thick] (-2,0) circle [radius=0.85];
    \node[below] at (0,0) {$0$};    % Label 0 on the x-axis
    \filldraw (0,0) circle (1pt);
    \draw[thick] (0,0) circle [radius=0.45];
    \draw[thick] (0,0) circle [radius=0.65];
    \draw[thick] (0,0) circle [radius=0.85];
\end{tikzpicture}
\end{center}
\caption{Contours for $m = 2$: The three circles on the left are $C^{\text{in}}_{2,L}$, $C_{1,L}$, $C^{\text{out}}_{2,L}$, and the three circles on the right are $C^{\text{in}}_{2,R}$, $C_{1,R}$, $C^{\text{out}}_{2,R}$. Each list is ordered from inside to outside.}
\label{fig:allContours}
\end{figure}
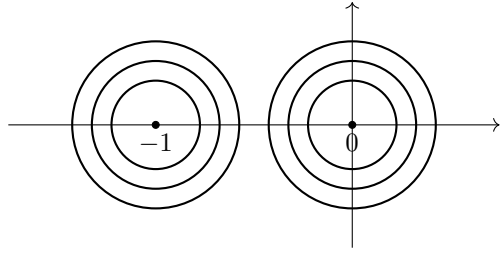

Let $C$ be a generic embedded curve in $\C$.

\phantomsection\label{not:delta-kernel}We define the delta kernel $\delta(x,y)$ on $C$ by the relation\[\int_{C} \delta(x,y) g(x) \frac{dx}{2\pi i} = g(y)\]
for any function $g \in L^2\big(C, \frac{dx}{2\pi \mathrm{i}}\big)$ and any $y \in C$. 

\phantomsection\label{not:delta-determinant}Let $\mb r, \mb s$ be vectors of dimension $m$. Then, we define a determinant of the above Dirac $\delta$-functions in a formal way by the Leibniz formula: \[\det[\delta(\mb r, \mb s)] = \det [\delta(r_i, s_j)]_{i,j= 1}^m = \sum_{\sigma \in S^m} \operatorname{sgn}(\sigma) \prod_{i = 1}^m \delta(r_{\sigma(i)}, s_i).\footnote{For the specialization to $\delta(-\eta_r^1-1,\xi_s^1)$ and its role
in Liu's flat initial-condition factor, see
Appendix~\ref{app:nc-flat-factor}.}\]

\subsection{Formula for Conditional Multi-Point Distributions}

In Proposition 2.3 of \cite{liu2021}, the authors present an explicit conditional multi-point distribution including both step and flat cases of TASEP. The authors of \cite{BCT} considered the case $I = \{1, \dots , m -1\}$ of this proposition and translated this result to exponential DLPP for the step initial condition using the well-known relationship (3.1) in \cite{BCT}.

We make the same translation for the flat case, using the parameter
correspondence in \eqref{eq:app-nc-tasep}. The computations for the flat
and step cases are given in Appendices~\ref{app:notation-flat}
and~\ref{app:notation-step}, respectively.

Before presenting the formula for the conditional multi-point distribution,
we define the following contour integral. The inner and outer contour
coefficients follow from \eqref{eq:app-nc-measure}. \begin{multline}
\label{eq:Dnflat}
 {\Sf D}^{(\mb n),\textup{flat}}_{\mb M, \mb N, \mb T}(z_1, \dots, z_{m-1}) = \frac{(-1)^{n_1}}{(2\pi i )^{2|\mb n|}} \prod_{i = 2}^m \prod_{k_i = 1}^{n_i} \lt[\int_{C_{i,L}^{in}} \dd\xi_{k_i}^i + z_{i-1} \int_{C_{i,L}^{out}} \dd\xi_{k_i}^i\rt]\lt[\int_{C_{i,R}^{in}} \dd\eta_{k_i}^i + z_{i-1} \int_{C_{i,R}^{out}} \dd\eta_{k_i}^i\rt] \\ 
 \times \prod_{k_1 = 1}^{n_1}\lt[\int_{C_{1,L}}\dd \xi_{k_1}^1\rt]\lt[\int_{C_{1,R}} \dd\eta_{k_1}^1\rt] \frac{\Pi_{\mb n}(\bs \xi, \bs \eta)}{K_{n_1}(\bs \xi^{1} | \bs \eta^{1})} \Sf F_{\mb M, \mb N, \mb T}^{(\mb n)} (\bs \xi, \bs \eta) \det \lt[\delta(-\eta_r^1 - 1, \xi_s^1)\rt]_{r,s =1}^{n_1}.
\end{multline}
\textit{Remark 1:} Unlike in \cite{BCT}, the radii of the contours are specifically chosen such that $C_{1,L} = -1 - C_{1,R}$.  

\textit{Remark 2:} ${\Sf D}^{(\mb n),\textup{flat}}_{\mb M, \mb N, \mb T}$ differs from $\Sf D_{\mb M, \mb N, \mb T}^{(\mb n)}$ by the $\delta$-function, as well as by a factor of $(-1)^{n_1}$, and the presence in the integrand of $\frac{1}{K_{n_1}(\bs \xi^1 | \bs \eta ^1)}$. These three changes are precisely the factor in
\eqref{eq:app-nc-flat-factor}.

Now we may state our analogue to Proposition 3.1 in \cite{BCT}.

\begin{prop}
\label{prop:multipointformula}
Consider the line-to-point exponential DLPP as in Definition~\ref{def:flat-explpp}. Fix an integer $m\ge 2$. 
Let $\mb M = (M_1,\cdots, M_m)\in \N^m$, $\mb N = (N_1,\cdots, N_m)\in \N^m$, and $\mb T = (T_1,\cdots, T_m)\in \R_+^m$. 
Assume that $0< T_1 \leq \cdots \leq T_m$  and $(N_1,T_1), \cdots, (N_m,T_m)$ are all distinct.  
Then, 
\begin{equation} 
\label{eq:conditional-ratio-main}
   	\p( \mathcal L^{\mathrm{flat}}(M_1,N_1) > T_1,\cdots , \mathcal L^{\mathrm{flat}}(M_{m-1},N_{m-1}) > T_{m-1}| \mathcal L^{\mathrm{flat}}(M_m,N_m) = T_m) 
    	= \frac{ \Sf Q^{\textup{flat}}_m(\mb M, \mb N, \mb T)}{\Sf Q^{\textup{flat}}_1(M_m, N_m, T_m)},
\end{equation}
where 
\begin{equation}
\label{eq:Qmflat-main}
    	\Sf Q^{\textup{flat}}_m(\mb M, \mb N, \mb T) = \sum_{\mb n\in\N^m}\frac{1}{(\mb n!)^2} \Sf Q^{(\mb n), \textup{flat}}_m(\mb M, \mb N,\mb T),
\end{equation}
with
\begin{equation} \label{eq:Qnflat}  
    	\Sf Q^{(\mb n), \textup{flat}}_m(\mb M, \mb N,\mb T) = \frac{(-1)^{|\mb n|+m-1}}{(2\pi \text{i})^{m-1}} \oint_{>1}\cdots \oint_{>1} 
   	\Sf D^{(\mb n),\textup{flat}}_{\mb M, \mb N, \mb T}(\mb z)  \prod_{i=1}^{m-1} \frac{(z_i+1)^{n_i-n_{i+1}-1}}{z_i^{n_{i+1}+1}} \dd z_i. 
\end{equation}
The function $\Sf D^{(\mb n),\textit{flat}}_{\mb M, \mb N, \mb T}(\mb z)$ is defined in \eqref{eq:Dnflat}, and the contours are circles centered at the origin with  radii greater than $1$.\footnote{The auxiliary-variable powers in \eqref{eq:Qnflat} are derived in
\eqref{eq:app-nc-D-conversion}--\eqref{eq:app-nc-mixed};
the density interpretation of $\Sf Q_m^{\mathrm{flat}}$, the restriction
to $\N^m$, and the one-point denominator are explained in
Appendix~\ref{app:nc-conditioning}.}
\end{prop}

\begin{proof}
The proof is nearly verbatim to that of Proposition 3.1 in \cite{BCT}. The only difference is that when applying Proposition 2.3 from \cite{liu2021}, instead of substituting and rearranging $\mathcal D_{Y_{\text{step}}}$, we use $\mathcal D_{Y_{\textup{flat}}}$ given in (2.19) in \cite{liu2021}. This translation is detailed in Appendix~\ref{app:notation-flat},
with the step comparison in Appendix~\ref{app:notation-step}.
In particular, \eqref{eq:app-nc-mixed} gives \begin{multline*}
 \p( \mathcal L^{\mathrm{flat}}(M_1,N_1) > T_1,..., \mathcal L^{\mathrm{flat}}(M_{m-1},N_{m-1}) > T_{m-1},\mathcal L^{\mathrm{flat}}(M_m,N_m) \leq T_m) \\  = \frac{1}{(2\pi \text{i})^{m-1}} 
    \sum_{\mb n\in\N_0^m}\frac{(-1)^{|\mb n|+m-1}}{(\mb n!)^2} 
    \oint_{>1}\cdots \oint_{>1} \wt{\Sf D}^{(\mb n), \textup{flat}}_{\mb M, \mb N, \mb T}(\mb z) 
    \prod_{i=1}^{m-1} \frac{(z_i+1)^{n_i-n_{i+1}-1}}{z_i^{n_{i+1}+1}} \dd z_i, 
\end{multline*}\phantomsection\label{not:undifferentiated-flat}
where $\wt{\Sf D}^{(\mb n), \textup{flat}}_{\mb M, \mb N, \mb T}(\mb z)$ is the same as ${\Sf D}^{(\mb n), \textup{flat}}_{\mb M, \mb N, \mb T}(\mb z)$, except that $\Pi_{\mb n}$ is replaced by $\wt \Pi_{\mb n}$, which is the same expression as $\Pi_{\mb n}$ but without the $S_{n_m}(\bs \xi^m, \bs \eta^m)$ term. The explicit definitions are \eqref{eq:app-nc-Pitilde} and
\eqref{eq:app-nc-Dtilde}. Differentiating as in
\eqref{eq:app-nc-D-derivative} and removing zero components of $\mb n$
as in Appendix~\ref{app:nc-conditioning} completes the proof.
\end{proof}

\subsection{Formula for a Special Case}
The special case  when $\mb n = (1, \dots, 1): = \mb 1$ will be useful when we consider the fluctuations along the diagonal in the proof of Theorem \ref{thm:cond-diag-fluc} in Section \ref{sec:diag-fluc}. This result parallels that of Lemma 3.3 in \cite{BCT}.

\begin{lemma}
We have  
    \begin{equation}{\label{eq:Q11}}
        \Sf Q^{(\mb 1), \textup{flat}}_{m}(\mb M, \mb N, \mb T)  = \int_{\vec \gamma} \frac{\dd \bs \xi}{(2\pi i)^{m}} \int_{\vec \Gamma} \frac{\dd \bs \eta}{(2\pi i)^{m}}\; \frac{\Pi_{\mb 1}(\bs \xi, \bs \eta)}{K_{1}( \xi^{1} | \eta^{1})} \Sf F_{\mb M, \mb N, \mb T}^{(\mb 1)} (\bs \xi, \bs \eta) \delta(-\eta^1 - 1, \xi^1),
    \end{equation}
    where $\bs\xi= (\xi^1, \cdots, \xi^m) \in \C^m$ and $\bs\eta= (\eta^1, \cdots, \eta^m) \in \C^m$. The contours are 
\begin{equation}
    \vec{\gamma} = \gamma_1\times \cdots \times \gamma_m, 
    \qquad
    \vec{\Gamma} = \Gamma_1\times \cdots\times \Gamma_m, 
\end{equation}
where $\gamma_1, \cdots, \gamma_m$ are small circles around the point $z=-1$, nested from inside to outside, and $\Gamma_1, \cdots, \Gamma_m$ are small circles around the point $z=0$, also nested from inside to outside, such that all circles are mutually disjoint. Furthermore, we require that $\gamma_1 = -1 - \Gamma_1$.
    
\end{lemma}
\begin{proof}
The proof method is identical to that of Lemma 3.3 in \cite{BCT} because the coupling induced by the $\delta$-function does not affect the order or total number of the $z_i$ terms. The only difference is that we have $\Sf D_{\mb M, \mb N, \mb T}^{(\mb n), \textup{flat}}$ instead of $\Sf D_{\mb M, \mb N, \mb T}^{(\mb n)}$. 
\end{proof}
\textit{Remark 1:} Comparing the result with Lemma 3.3 in \cite{BCT}, we see that our integrand matches the integrand of $\Sf D_{\mb M, \mb N, \mb T}^{(\mb n), \textup{flat}}$, and the aforementioned factor of $(-1)^{n_1}$ contributed a $-$ sign. The determinant of the $\delta$-kernels simplified to a single $\delta$-function.

\textit{Remark 2:} Unlike \cite{BCT}, we once again required  $\gamma_1 = -1 - \Gamma_1$ so that the $\delta$-function made sense.

\section{Proof of Theorem \ref{thm:cond-diag-fluc}}
\label{sec:diag-fluc}

Without loss of generality, we change coordinates such that the conditioned point is located on the diagonal, that is, we assume $a = b$. Also note throughout that we are using $L$ as the large parameter, whereas in the Section~\ref{sec:intro} we used $N$. 

To prove the theorem, we show that for every $m\ge 2$,
\begin{equation} \label{eq:diagpr}
    \p \left( \bigcap_{j=1}^{m-1} \left\{  \frac{ \mathcal L^{\mathrm{flat}}( t_j a L +  x_j\frac{a}{\sqrt{D}} \sigma  L^{1/2}, t_j aL -  x_j\frac{a}{\sqrt{D}} \sigma  L^{1/2})- t_j \ell L}{\sigma L^{1/2}}  > h_j \right\}\, \bigg| \, \mathcal L^{\mathrm{flat}}(aL, aL)= \ell L \right)  
\end{equation}
converges, as $L\to \infty$, to 
\begin{equation}
	\Sf P^{\textup{flat}}(\mb t, \mb x, \mb h):= \p\left( \bigcap_{j=1}^{m-1} \left\{ \mathbb B_1(t_j) - \left| \mathbb B_2(t_j)-(1-t_j)Z - x_j \right|  \} > h_j\right\} \right)
\end{equation}
for every 
\begin{equation} \label{eq:txhvecd}
    \mb t=(t_1, \cdots,t_{m-1})\in (0,1)^{m-1}, 
    \quad \mb x=(x_1, \cdots,x_{m-1})\in \R^{m-1}, 
    \quad \mb h=(h_1, \cdots,h_{m-1})\in \R^{m-1}. 
\end{equation}
 $\mathbb B_1$ and $\mathbb B_2$ are independent Brownian bridges and $Z \sim \mathcal N(0,1)$.

 Using the identity $\min(p,q)= \frac{p+q}{2}- \frac{|p-q|}{2}$, we may write $\Sf P^{\textup{flat}}(\mb t, \mb x, \mb h)$ in a form similar to that of \cite{BCT} as \begin{equation}
\p\lt( \bigcap_{j=1}^{m-1} \Big\{ \min\big\{ \sqrt 2 \mathbb B_+^{\mathrm{br}}(t_j) + (1-t_j)Z + x_j, \sqrt 2 \mathbb B^{\mathrm{br}}_-(t_j) - (1-t_j)Z - x_j \big\} > h_j \Big\} \rt).
 \end{equation}

Now fix $m \geq 2$, $\mb t, \mb x, \mb h$. We may apply Lemma 4.2 from \cite{BCT} successively\footnote{This is explained in more detail in \cite{BCT} Section 6.} so that we may assume \[
0 < t_1 < \cdots < t_{m-1} < 1.
\]
 To state the lemmas that prove the theorem, we must first reproduce some more notation from \cite{BCT} Section 6.1. 

\subsection{Notation}
Let \[\mb M_L = (M_{L,1}, \cdots, M_{L,{m}})\in \N^m, \qquad \mb N_L = (N_{L,1}, \cdots, N_{L,{m}})\in \N^m, \qquad \mb T_L = (T_{L,1}, \cdots, T_{L,{m}}) \in \R_+^m,\]
in which  \footnote{ $\lceil \; \cdot \; \rceil$ is the ceiling function.} 
\begin{equation}
\label{eq:defMLi}
    M_{L,i} = \bigg\lceil t_i  a L + x_i \frac{a\sigma}{ \sqrt{D}}  L^{1/2} \bigg \rceil, 
    \quad  N_{L, i} =  \bigg\lceil  t_i a L - x_i \frac{ a\sigma}{ \sqrt{D}} L^{1/2}  \bigg\rceil , 
    \quad T_{L, i} = t_i \ell L + h_i \sigma  L^{1/2}, 
\end{equation}
\[
    M_{L,m}= \lceil aL \rceil, \qquad N_{L,m}=\lceil aL \rceil, \qquad T_{L,m}= \ell L. 
\]
We also set $M_{L,0}=N_{L,0}=T_{L,0}=0$. \footnote{The formulae above differ slightly from \cite{BCT} Section 6 since we have been assuming, without loss of generality, $a= b$.}

For $L>0$, define the constant \begin{equation}
 \label{eq:cKdef}
    \Sf Z_L:= \left(\frac{\ell+\sqrt{D}} {\ell -\sqrt{D}}  \right)^{2\lceil aL \rceil} 
    e^{- \sqrt{D} L},
\end{equation}
where $D$ is from \eqref{eq:D}.
\subsection{Statement of Lemmas} 

 As discussed in \cite{BCT} Section 6, Proposition \ref{prop:multipointformula} implies that Theorem \ref{thm:cond-diag-fluc} is proven if we show that 
 \[\lim_{L \to \infty}\frac{ \Sf Q^{\textup{flat}}_m(\mb M_L, \mb N_L, \mb T_L)}{\Sf Q^{\textup{flat}}_1(M_{L,m}, N_{L,m}, T_{L,m})} = \Sf P^{\textup{flat}} (\bf t, \bf x, \bf h). \]
 To this end, we prove three similar lemmas to 6.1, 6.2, 6.3 in \cite{BCT}. 

The first lemma shows the $\mb n = \mb 1$ term dominates. 
\begin{lemma}
\label{lem:dominatingterm}
   \[
\lim_{L \to \infty} \frac{\sqrt{8\pi} \sigma L^{1/2}}{\Sf Z_L} \Sf Q^{(\mb 1), \textup{flat}}_m (\mb M_L, \mb N_L, \mb T_L) = \Sf P^{\textup{flat}}(\mb t, \mb x, \mb h).
\]
\end{lemma}

The second lemma shows the rest of the terms vanish in the limit.
\begin{lemma}\label{lem:restofterms}
   There exists a constant $c>0$ such that
\begin{equation}
    \frac1{\Sf Z_{L}} \sum_{\mb n\in\N^m\setminus\{\mb 1\}}\frac1{(\mb n!)^2}{\Sf Q_m^{(\mb n), \textup{flat}}(\mb M_L,\mb N_L,\mb T_L)} \leq e^{-c L}
\end{equation}
for all sufficiently large $L>0$.
\end{lemma}
\begin{lemma}
\label{lem:trivial}
    \[\lim_{L \to \infty}\frac{\sqrt{8\pi} \sigma L^{1/2}}{\Sf Z_L} \Sf Q_{1}^{\textup{flat}}(aL, aL , \ell L) = 1\]
\end{lemma}

These three lemmas give the proof of Theorem \ref{thm:cond-diag-fluc}. We will first introduce some notation and discuss a few important properties in Subsection \ref{sec:notation}. Then,
we prove Lemmas \ref{lem:dominatingterm} and \ref{lem:restofterms} in Subsections \ref{sec:proofdominatingterm} and \ref{sec:proofrestofterms}, respectively.
Lemma \ref{lem:trivial} directly follows from these two lemmas. 

\subsection{Summary of Lemma \ref{lem:dominatingterm}}
The authors of \cite{BCT} proved their Lemma 6.1 via application of the method of steepest descent  to $\Sf F^{(\mb 1)}(\bs \xi, \bs \eta) $ after asymptotically expanding $\mb M, \mb N, \mb T$ in terms of the large parameter $L$. We use the same method and the same asymptotic expansion of each of these terms. Details of the asymptotic expansion, steepest descent, and proof method may be found in Sections 5 and 6.1 of \cite{BCT}. Importantly, we have the same critical points $z^\pm$, and we will deform our contours in the same way, so much of the argument is essentially the same.

The main differences appear in our handling of the coupling induced by the $\delta$-function.

For clarity, in the following subsection we will reproduce, nearly verbatim, the notation used in Section 6.1 of \cite{BCT}. As mentioned, we  let $a = b$ in those formulae.

\subsection{Notation and Properties for Lemma \ref{lem:dominatingterm}}
\label{sec:notation}

The quantity $\Sf Q_m^{(\mb n), \textup{flat}}$ in \eqref{eq:Qnflat} is expressed in terms of $\Sf D^{(\mb n), \textup{flat}}_{\mb M_L, \mb N_L, \mb T_L}(\mb z)$ from  \eqref{eq:Dnflat}, in which we have the functions
\begin{equation}
    \Sf f_{L,i}(z):= \frac{f_{M_{L,i}, N_{L,i},T_{L,i}}(z)}{f_{M_{L,i-1},N_{L,i-1},T_{L,i-1}}(z)}, 
    \qquad
    f_{M,N,T}(z)=  \frac{z^Ne^{Tz}}{(z+1)^M} = e^{- M\log(z+1) +N\log z+ Tz}. 
\end{equation}
From \eqref{eq:defMLi}, we have 
\begin{equation}\label{eq:j_scaling2}
    M_{L,i} =  t_i  a L + x_i \frac{a\sigma}{ \sqrt{D}}  L^{1/2}  + \epsilon_{L, i}
    \qquad  
    N_{L, i} =  t_i a L - x_i \frac{ a\sigma}{  \sqrt{D}} L^{1/2}  + \epsilon'_{L, i}
\end{equation}
for real numbers $ \epsilon_{L, i},  \epsilon_{L, i}'\in [0,1)$. 
Thus, 
\begin{equation} \label{eq:fLide}
    \Sf f_{L,i}(z) = e^{(t_i-t_{i-1}) G(z) L + H_i(z) L^{1/2} + E_{L, i}(z) } 
\end{equation}
where
\begin{equation} \label{eq:GGHHE} \begin{split}
    &G(z)= -a\log(z+1)+a\log z +\ell z, \\
    &H_i(z)= - (x_i-x_{i-1}) \frac{a\sigma}{  \sqrt{D}} \log(z+1) - (x_i-x_{i-1}) \frac{ a\sigma}{  \sqrt{D}} \log z + (h_i-h_{i-1}) \sigma z, \\
    &E_{L, i}(z)= - \delta^1_{L,i}\log (z+1)+\delta^2_{L,i} \log z, 
\end{split} \end{equation}
with real numbers 
\begin{equation}
    \delta^1_{L,i}, \, \delta^2_{L,i} \in (-1, 1). 
\end{equation}
Here, we set $x_0=t_0=h_0=x_m=h_m=0$ and $t_m=1$. 
The $\Sf f_{L,i}(z)$ are analytic except possibly  at $z=0$ and $z= -1$. 

We now list a few properties:

\begin{itemize}
\item 
From Lemma 5.2 in \cite{BCT}, the critical points of $G$ are
\begin{equation} \label{eq:c3+-}
    z^{\pm} = \frac{- \ell  \pm \sqrt{D}}{2\ell}, \qquad D=\ell^2-4a\ell, 
\end{equation}
and they satisfy the inequalities $-1<z^-<z^+<0$. 
Importantly, we will note that $a = b$ forces $-1 - z^+ = z^-$. This proves to be crucial as $\delta$ couples $-1-\eta^1$ to $\xi^1$.
\item 
One may check  (see 5.5 in \cite{BCT}) that 
\begin{equation} \label{eq:GGpaHp}
    G''(z^\pm)= \mp\frac{\sqrt{D}\ell}{a} =  \mp 2 \sigma^2, \quad H_i'(z^\pm)=  \sigma \left( \pm (x_i-x_{i-1}) + h_i-h_{i-1} \right), \quad \Sf c_\pm = 1.\footnote{An expression for $ \Sf c_\pm$ when $a \neq b$ may be found in (1.14) of \cite{BCT})}
\end{equation}

\item 
Since $\prod_{i=1}^m \Sf f_{L,i}(z)= f_{M_{L,m}, N_{L,m},T_{L,m}}(z)$, we have that
\begin{equation} \label{eq:finzpm}
    \prod_{i=1}^m \frac{\Sf f_{L,i}(z^-)}{ \Sf f_{L,i}(z^+)}
    = \frac{f_{M_{L,m}, N_{L,m},T_{L,m}}(z^-)}{f_{M_{L,m}, N_{L,m},T_{L,m}}(z^+)} 
    = \Sf Z_L, 
\end{equation}
where $\Sf Z_L$ is as in \eqref{eq:cKdef}. 
\item 
It is a direct computation that
\begin{equation} 
    G(z^+)-G(z^-)= \sqrt{D} + 2a \log \Bigg( \frac{\ell-\sqrt{D}}{\ell+\sqrt{D}}\Bigg)  =:\Sf J(\ell), 
\end{equation}
In particular, $\Sf J(\ell)>0$. 
Since $t_i-t_{i-1}>0$ for every $i$, there exists $L_0>0$ such that 
\begin{equation} \label{eq:frati}
    \frac{\Sf f_{L, i}(z^-)}{\Sf f_{L, i}(z^+)} \le e^{-\frac{1}{2} (t_i-t_{i-1}) (G(z^+)- G(z^-)) L}  \le  e^{- \frac{1}{2} \tau J(\ell) L} , \quad 
    \tau := \min_{1 \le i \le m} (t_i - t_{i-1}) > 0, 
\end{equation}
for every $L>L_0$ and $i=1, \dots , m$. 
\end{itemize}

\subsection{Proof of Lemma \ref{lem:dominatingterm}}
\label{sec:proofdominatingterm}

 Recall Equation \eqref{eq:Q11}: \begin{equation}{\label{eq:Q11L}}
 \Sf Q^{(\mb 1), \textup{flat}}_{m}(\mb M_L, \mb N_L, \mb T_L)  = \int_{\vec \gamma} \frac{\dd \bs \xi}{(2\pi i)^{m}} \int_{\vec \Gamma} \frac{\dd \bs \eta}{(2\pi i)^{m}}\; \frac{\Pi_{\mb 1}(\bs \xi, \bs \eta)}{K_{1}(\xi^{1} | \eta^{1})} \Sf F_{\mb M_L, \mb N_L, \mb T_L}^{(\mb 1)} (\bs \xi, \bs \eta) \delta(-\eta^1 - 1, \xi^1).
\end{equation}
Using Cauchy's theorem, we may deform the contours --      without changing the value of the integral -- to\[
 \vec{\gamma}= \Sigma_-^{1,L}\times \cdots \times \Sigma_{-}^{m,L},
    \quad 
    \vec{\Gamma} =  \Sigma_{+}^{1,L} \times \cdots \times \Sigma_{+}^{m, L},
\]
where $\Sigma^{i, L}_{\pm}$ are as in Section 6.2 of \cite{BCT}. That is, 
\[
 \Sigma^{i, L}_{-}= \{ z\in \C \, : \, |z+1|= | z^- +1|+b L^{-\frac12} \}, 
    \qquad
    \Sigma^{i, L}_{+}= \{ z\in \C \, : \, |z|= |z^+| +b L^{-\frac12} \}. 
\]
Notice that $\Sigma^{1,L}_+$ and $\Sigma^{1,L}_{-}$ satisfy $-1 - \Sigma^{1,L}_+ = \Sigma^{1,L}_{-}$ so that the $\delta$-function still makes sense. 

Fix $\epsilon\in (0,1/2)$ and define $D_{L,-}^{\epsilon}$ and $D_{L,+}^{\epsilon}$ as in Section 6.2 of \cite{BCT}: \[
D_{L,-}^{\epsilon}=\{z\in \C : |z-z^-|\le L^{-\frac12+\frac{\epsilon}{3} }\}, \qquad D_{L,+}^{\epsilon}=\{z\in \C : |z-z^+|\le L^{-\frac12+\frac{\epsilon}{3} }\}
\] Set 
\[
    \vec{\gamma}^\epsilon= (\Sigma_{-}^{1,L}\cap D_{L,-}^\epsilon) \times \cdots \times (\Sigma_{-}^{m,L}\cap D_{L,-}^\epsilon), 
    \qquad
    \vec{\Gamma}^\epsilon =  (\Sigma_{+}^{1,L}\cap D_{L,+}^\epsilon) \times \cdots \times (\Sigma_{+}^{m, L}\cap D_{L,+}^\epsilon). 
\]
\subsubsection{Bounding Away from the Critical Points}
We may rewrite the rational function:
\begin{align}
\label{eq:Piexpanded}
\frac{\Pi_{\mb 1}(\bs \xi, \bs \eta)}{K_{1}(\xi^{1} | \eta^{1})} &= \lt(\xi^1 - \eta^1\rt)\frac{(-1)^{m}}{\xi^{m}-\eta^{m}} \prod_{i=1}^{m-1} \frac{(\xi^{i}-\eta^{i+1})(\eta^{i}-\xi^{i+1})}{(\xi^{i}-\xi^{i+1})(\eta^{i}-\eta^{i+1})(\xi^{i}-\eta^{i})^2} \notag\\
& = \underbrace{\frac{(-1)^{m}}{\xi^{m}-\eta^{m}} \prod_{i=2}^{m-1} \frac{(\xi^{i}-\eta^{i+1})(\eta^{i}-\xi^{i+1})}{(\xi^{i}-\xi^{i+1})(\eta^{i}-\eta^{i+1})(\xi^{i}-\eta^{i})^2}}_{\Pi_\mb 1^*(\bs \xi, \bs \eta)}\underbrace{\lt[\frac{(\xi^{1}-\eta^{2})(\eta^{1}-\xi^{2})}{(\xi^{1}-\xi^{2})(\eta^{1}-\eta^{2})(\xi^{1}-\eta^{1})}\rt]}_{W(\eta^1, \eta^2, \xi^1, \xi^2)}.
\end{align}
We define $\Pi_\mb 1^*(\bs \xi, \bs \eta)$, and $W(\eta^1, \eta^2, \xi^1, \xi^2)$ as indicated above. Similarly, we expand \begin{equation}
\label{eq:Fexpanded}
\Sf F^{(\mb 1)}_{\mb M_L, \mb N_L, \mb T_L} (\bs\xi, \bs\eta) 
    =  \prod_{i=1}^m  \frac{\Sf f_{L,i}(\xi^{i})}{\Sf f_{L,i}(\eta^{i})} =  \frac{\Sf f_{L,1}(\xi^1)}{\Sf f_{L,1}(\eta^1)}\underbrace{ \prod_{i=2}^m  \frac{\Sf f_{L,i}(\xi^{i})}{\Sf f_{L,i}(\eta^{i})}}_{\Sf F^{(\mb 1),*}_{\mb M_L, \mb N_L, \mb T_L} (\bs\xi, \bs\eta) }.
\end{equation}
We define $\Sf F^{(\mb 1),*}_{\mb M_L, \mb N_L, \mb T_L} (\bs\xi, \bs\eta)$ as indicated above. Substituting \eqref{eq:Piexpanded} and  \eqref{eq:Fexpanded}, and evaluating the $\delta$-function, \eqref{eq:Q11L} becomes \begin{multline}
\label{eq:Qexpanded}
\Sf Q^{(\mb 1), \textup{flat}}_{m}(\mb M_L, \mb N_L, \mb T_L)  =   \int_{\Sigma_+^{1,L}}\frac{\dd\eta^1}{2\pi i }\int_{\Sigma_-^{2,L} \times \cdots \times \Sigma_{-}^{m,L}}\frac{\dd \bs \xi^*}{(2\pi i )^{m-1}} \int_{\Sigma_+^{2,L} \times \cdots \times \Sigma_{+}^{m,L}} \frac{\dd \bs \eta^*}{(2\pi i )^{m-1}}\; \\\Pi_\mb 1^*(\bs \xi ^*, \bs \eta ^*)\Sf F_{\mb M_L, \mb N_L, \mb T_L}^{(\mb 1),*} (\bs \xi^*, \bs \eta^*) \times
 \lt[\frac{\Sf f_{L,1}(-1 - \eta^1)}{\Sf f_{L,1}(\eta^1)} W(\eta^1, \eta^2, -1 - \eta^1, \xi^2)\rt],
\end{multline}
where the $*$, here and going forward, denotes omission of the $i = 1$ index. 
 
We consider \begin{multline*}
\int_{\Sigma_+^{1,L} \setminus D^{\eps}_{L,+}}\frac{\dd\eta^1}{2\pi i} \int_{\vec \gamma^* \setminus \vec \gamma^{\eps,*}} \frac{\dd \bs \xi^*}{(2\pi i)^{m-1}} \int_{\vec \Gamma^* \setminus \vec \Gamma^{\eps,*}} \frac{\dd \bs \eta^*}{(2\pi i)^{m-1}}\; \Pi_\mb 1^*(\bs \xi ^*, \bs \eta ^*)\Sf F_{\mb M_L, \mb N_L, \mb T_L}^{(\mb 1),*} (\bs \xi ^*, \bs \eta ^*) \\\times
\lt[\frac{\Sf f_{L,1}(-1 - \eta^1)}{\Sf f_{L,1}(\eta^1)} W(\eta^1, \eta^2, -1 - \eta^1, \xi^2)\rt].
\end{multline*}
 Lemma 5.3(a) and Lemma 5.3(b) from \cite{BCT} bound the ratios $\frac{\Sf f_L(z)}{\Sf f_L(z^-)}$ and $\frac{\Sf f_L(z^+)}{\Sf f_L(z)}$ on the region we are considering. Since $a > 0$, and $\ell > 2(a+a) =(\sqrt{a}+ \sqrt{a})^2$, these lemmas apply directly to $\Sf f_{L,i}(z)$, $i \geq 2$. As mentioned earlier, $-1 - \Sigma^{1,L}_+ = \Sigma^{1,L}_{-}$, and $-1 - z^+ = z^-$. We can also see that when one writes \[
\Sf f_{L,1}(-1 -\eta^1)  = e^{t_1G(-1 - \eta^1)L + H_1(-1-\eta^1)L^{1/2} + E_{L,1}(-1-\eta^1)}
\]
that the critical points of the function $G(-1 - z)$ are the same as that of $G(z)$. Thus, $\Sf f_{L,1}(-1 - \eta^1)$ also satisfies the same Lemma 5.3(a) from \cite{BCT}. As in Section 6.1 in \cite{BCT}, we have that \[
\frac{\Sf f_{L,1}(z^-)}{ \Sf f_{L,1}(z^+)}\underbrace{\prod_{i=2}^m \frac{\Sf f_{L,i}(z^-)}{ \Sf f_{L,i}(z^+)}}_{\Sf Z_L^*} = \Sf Z_L.
\]
$\Sf Z_L^*$ is defined as indicated above. As in Section 6.2 in \cite{BCT}, applying Lemmas 5.3(a) and 5.3(b) to each ratio shows that \begin{equation}
\label{eq:unifF}
\frac{|\Sf F^{*}_{\mb M_L, \mb N_L, \mb T_L}|}{\Sf Z_L^*} = O\big(e^{-cL^{2\eps/3}}\big)  \text{ uniformly  for } (\bs \xi^*, \bs \eta^*) \in (\vec \gamma^* \times \vec \Gamma^* ) \setminus (\vec \gamma^{\eps,*} \times \vec \Gamma^{\eps,*}), \text{ as } L \to \infty.
\end{equation}
Similarly, we have that \text{as} $L \to \infty$, \begin{equation}
\label{eq:unifFF}
 \frac{\Sf f_{L,1}(-1 - \eta^1)}{ \Sf f_{L,1}(z^-)}  \frac{\Sf f_{L,1}(z^+)}{\Sf f_{L,1}(\eta^1)} = O\big(e^{-cL^{2\epsilon/3}}\big)\text{ uniformly  for } (\xi^1, \eta^1) \in (\Sigma_-^{1,L} \setminus D^\eps_{L,-})\times (\vec \Sigma_+^{1,L} \setminus D_{L,+}^\eps).
\end{equation}
One bounds $\Pi_\mb 1^*(\bs \xi^*, \bs \eta^*)$ in exactly the same method as $\Pi_\mb 1(\bs \xi, \bs \eta)$ is bounded in  \cite{BCT} Section 6.2, as we are considering the same contours and variables on those contours, except for the $\eta^1$ and $\xi^1$ components. In \cite{BCT}, the authors had
\[\Pi_\mb 1(\bs \xi, \bs \eta) = O(L^{m-1}) \quad \text{uniformly for } (\bs \xi, \bs \eta) \in (\vec \gamma \times \vec \Gamma) \setminus (\vec \gamma^\eps \times \vec \Gamma^\eps), \; \text{ as } L \to \infty.\]
Because we are excluding the first term in the product of $\Pi_\mb1 (\bs \xi, \bs \eta)$ to form $\Pi_\mb1^* (\bs \xi^*, \bs \eta^*)$, we instead have\begin{equation}
\label{eq:unifPi}
\Pi_\mb 1^*(\bs \xi^*, \bs \eta^*) = O(L^{m-2}) \quad \text{uniformly for } (\bs \xi^*, \bs \eta^*) \in (\vec \gamma^*\times \vec \Gamma^*) \setminus (\vec \gamma^{\eps,*} \times \vec \Gamma^{\eps,*}), \; \text{ as } L \to \infty.\end{equation}
Again using this method, we see that for $W(\eta^1,\eta^2, -1 - \eta^1, \xi^2)$,
as $L \to \infty$, \begin{equation}
\label{eq:unifW}
W(\eta^1,\eta^2, -1 - \eta^1, \xi^2)= O(L) \text{ uniformly for } (\xi^2, \eta^1, \eta^2) \in   (\vec \Sigma_-^{2,L} \setminus D^\eps_{L,-}) \times (\vec \Sigma_+^{1,L} \setminus D^\eps_{L,+}) \times (\vec \Sigma_+^{2,L} \setminus D^\eps_{L,+}).
\end{equation}
Thus, similar to (6.19) in \cite{BCT}, we have, using \eqref{eq:Qexpanded}, \eqref{eq:unifF}, \eqref{eq:unifFF}, \eqref{eq:unifPi}, and \eqref{eq:unifW}: \[
\frac{1}{\Sf Z_L L^{m-1} }  
    \int_{ (\vec{\gamma}\times \vec{\Gamma}) \setminus (\vec{\gamma}^\epsilon \times \vec{\Gamma}^\epsilon) } \dd \bs \xi  \, \dd \bs \eta \, 
    \frac{\Pi_{\mb 1}(\bs \xi, \bs \eta)}{K_{1}(\xi^{1} | \eta^{1})} \Sf F_{\mb M_L, \mb N_L, \mb T_L}^{(\mb 1)} (\bs \xi, \bs \eta) \delta(-\eta^1 - 1, \xi^1)
    = O(e^{-cL^{2\epsilon/3}}).  
\]
 \subsubsection{Change of Variables Near the Critical Points}
Now we consider \begin{equation}
\label{eq:startofassembly}\int_{\vec \gamma^\eps} \dd \bs \xi \int_{\vec \Gamma^\eps} \dd \bs \eta\; \frac{\Pi_{\mb 1}(\bs \xi, \bs \eta)}{K_{1}(\xi^{1} | \eta^{1})} \Sf F_{\mb M_L, \mb N_L, \mb T_L}^{(\mb n)} (\bs \xi, \bs \eta) \delta(-\eta^1 - 1, \xi^1).
\end{equation}
Consider the change of variables given in \cite{BCT}: \[
\xi^i = z^- + \frac{u_i}{\sigma L^{1/2}} =: \ph(u_i), \quad \eta^i = z^+ - \frac{v_i}{\sigma L^{1/2}} =: \psi(v_i).
\]
By textbook theory of distributions we obtain \begin{equation}
\label{eq:finaldeltapiece}
   \sigma^{-1} L^{-1/2} \ph^* \delta(-1-\psi(v_1),u_1) = \delta(v - (1 + z^+ + z^-) \sigma L^{1/2},u_1) = \delta(v_1 ,u_1) 
\end{equation}
since $a = b$, so we have $1 + z^+  + z^- = 0$. We now examine the expressions of the rest of the integrand after, in the specified order, \begin{enumerate}
    \item Changing variables,
    \item Coupling $u_1 = v_1$ via $\delta(v_1,u_1)$,
    \item Taking the limit as $L \to \infty$.
\end{enumerate}
Under change of coordinates, \[K_1(u_1,v_1)  = \frac{\sigma L^{1/2}}{\sigma L ^{1/2}(z^- - z^+)+ u_1 + v_1}.  \]
After evaluating the $\delta$-function coupling $u_1 = v_1$, and substituting in for $z^- - z^+$ from \eqref{eq:c3+-}, \[
K_1(v_1) = \frac{\sigma L^{1/2}}{- \frac{\sigma\sqrt{D}}{\ell}L^{1/2} + 2v_1} = \frac{\sigma L^{1/2}}{- \frac{\sigma \sqrt{D}}{\ell}L^{1/2}} \lt(1 + O (L^{-1/2})\rt) = - \frac{\ell}{\sqrt{D}}\lt(1 + O (L^{-1/2})\rt)
\]
with the second equality via a geometric series expansion  for sufficiently large $L$. 
As in Section 6.2 in \cite{BCT}, we make the change of variables for $\Pi_\mb1 (\bs \xi, \bs \eta)$  and substitute in for $z^+ - z^-$ from \eqref{eq:c3+-} to obtain $\hat\Pi_{\mb 1}(\mb u, \mb v)$ as below.
\[
\hat\Pi_{\mb 1}(\mb u, \mb v) = \frac{\ell(\sigma^2 L)^{m-1}}{\sqrt D}\lt[\prod_{i=1}^{m-1} \frac{1}{(u_i - u_{i+1})(v_{i+1} - v_i)}\rt]\lt(1 + O(L^{-\frac12 + \frac{\eps}{3}})\rt).
\]
As $L \to \infty$, division by $K_1$ multiplies this expression by $-\sqrt{D}/\ell$. Thus, 
\begin{equation}
\label{eq:finalPipiece}\frac{\hat\Pi_{\mb 1}(\mb u,\mb v)}{K_1(u_1,v_1)} = -\,(\sigma^2 L)^{m-1}\lt[\prod_{i=1}^{m-1}\frac{1}{(u_i-u_{i+1})(v_{i+1}-v_i)}\rt]\lt(1 + O(L^{-\frac12+\frac{\eps}{3}})\rt).\end{equation}
Note the factor $(\sigma^2 L)^{m-1}$ is the same as in the step case \cite{BCT}, and this term is asymptotically identical to $\hat\Pi_{\mb 1}(\mb u,\mb v)$ without dividing by $K_1$. 

Now make the change of variables into $\Sf F_{\mb M_L, \mb N_L, \mb T_L}^{(\mb n)} (\bs \xi, \bs \eta)$, and let this function be named $\hat{\Sf F}_L(\mb u, \mb v)$.

Note $\epsilon$ is still in $(0,1/2)$. Lemma 5.1 from \cite{BCT} allows us to write $\Sf f_L (z_c + wL^{-1/2})$ as \[
\Sf f_L(z_c) \exp\bigg(\frac{1}{2}G''(z_c) w^2 + H'(z_c) w \bigg)\bigg( 1 + O(L^{-1/2 + \epsilon})\bigg)
\] for sufficiently large $L$, $|w| \leq L^{\epsilon/3}$. We use this fact exactly as it is used in  Section 6.2 of \cite{BCT} and obtain  \begin{equation}
\label{eq:nearcriticalpointsF}
 \frac{\hat{\Sf F}_L(\mb u, \mb{v})}{\Sf Z_L}
= \Sf F(\mathbf{u}, \mathbf{v}) \lt(1 + O(L^{-1/2+\epsilon}) \rt)
    \quad \text{for $(\mathbf{u}, \mathbf{v})\in \vec{\Sigma}^-_L\times \vec{\Sigma}^+_L$,}\end{equation}
where
\begin{equation}
\Sf F(\mb u, \mb v) = \prod_{i=1}^m e^{(t_i - t_{i-1})u_i^2 + [-(x_i - x_{i-1}) + (h_i - h_{i-1})]u_i}\; e^{(t_i - t_{i-1})v_i^2 + [(x_i - x_{i-1}) + (h_i - h_{i-1})]v_i},
\end{equation}
and $\vec \Sigma^-_L \times \vec \Sigma_L^+$ are the images of $\vec \gamma^\eps$ and $\vec \Gamma^\eps$ under the change of variables. We note here the distinction from the similar expression in \cite{BCT} in that we have $\Sf c_\pm  =1$ from \eqref{eq:GGpaHp}. 

Coupling $u_1 = v_1$ under the $\delta$-kernel, then taking the limit as $L \to \infty$ yields \begin{equation}
\label{eq:finalFpiece}
   \exp(2t_1v_1^2 + 2h_1 v_1)\Sf  F^*(\mb u^*, \mb v^*),  
\end{equation}
in which $*$ denotes the omission of the $i = 1$ term, so \[
\Sf  F^*(\mb u^*, \mb v^*) : = \prod_{i=2}^m e^{(t_i - t_{i-1})u_i^2 + [-(x_i - x_{i-1}) + (h_i - h_{i-1})]u_i}\; e^{(t_i - t_{i-1})v_i^2 + [(x_i - x_{i-1}) + (h_i - h_{i-1})]v_i}.
\]
One notes, as in Section 6.2 of \cite{BCT}, that in the limit, $\Sf F^*(\mathbf{u^*}, \mathbf{v^*})$, $\exp(2t_1v_1^2 + 2h_1 v_1)$ each have super-exponential decay  in any closed sector strictly contained in $\{z\in \C: \arg (z)\in (\frac{\pi}{4}, \frac{3\pi}{4})\cup (-\frac{3\pi}{4}, -\frac{\pi}4)\}$.

Putting together \eqref{eq:startofassembly}, \eqref{eq:finaldeltapiece}, \eqref{eq:finalPipiece}, \eqref{eq:finalFpiece}, along with applying dominated convergence, as described in Section 6.2 of \cite{BCT}, we have the following: 
\begin{align}
\label{eq:covintegral}
   &\lim_{L \to \infty}\frac{\sqrt{8\pi} \sigma L^{1/2}}{\Sf Z_L}\int_{\vec \gamma^\eps} \frac{\dd \bs \xi}{(2\pi i )^m}\int_{\vec \Gamma^\eps} \frac{\dd \bs \eta}{(2\pi i )^m}\; \frac{\Pi_{\mb 1}(\bs \xi, \bs \eta)}{K_{1}(\xi^{1} | \eta^{1})} \Sf F_{\mb M_L, \mb N_L, \mb T_L}^{(\mb n)} (\bs \xi, \bs \eta) \delta(-\eta^1 - 1, \xi^1)\\
   &\qquad\qquad\qquad\qquad\qquad\qquad\qquad= \lim_{L \to \infty}\frac{\sqrt{8\pi}}{(\sigma^2 L)^{m-1} \Sf Z_L }   \int_{{\vec{\Sigma}^{-}}} \frac{\dd\mb u}{(2\pi i)^{m}}\int_{{\vec{\Sigma}^{+}}} \frac{\dd\mb v}{(2\pi i )^m} \frac{\hat\Pi_{\mb 1}(\mb u,\mb v)}{K_1(u_1,v_1)} \Sf {\hat{F}_L}(\mb u, \mb v)  \delta(v_1, u_1) \notag\\
   &\qquad\qquad\qquad\qquad\qquad\qquad\qquad= - \frac{\sqrt{8\pi}}{(\sigma^2L)^{m-1}} \int_{1 + i\R } \frac{\dd v_1}{2\pi i}\int_{{\vec{\Sigma}^{-}}*} \frac{\dd\mb u^*}{(2\pi i )^{m-1}} \int_{{\vec{\Sigma}^{+}*}}  \frac{\dd\mb v^*}{(2\pi i )^{m-1}} \notag\\
   &\qquad\qquad\qquad\qquad\qquad\qquad\qquad\qquad \times \lt[\prod_{i=2}^{m-1}\frac{1}{(u_i-u_{i+1})(v_{i+1}-v_i)}\rt]\lt[\frac{\exp(2t_1v_1^2 + 2h_1 v_1)}{(u_1-u_2)(v_2-v_{1})}\rt] \Sf F^*(\mb u^*, \mb v^*) \notag.
\end{align}
As usual, $*$ denotes the absence of the 1-index.  Lastly, note that as described in Section 6.2 of \cite{BCT}, each contour becomes a vertical line as $L \to \infty$. 

\subsubsection{Evaluating the Integral}
\eqref{eq:covintegral} is nearly analogous to (4.3) in \cite{liu2022}, from which point this integral is evaluated. However, there are many differences in convention, so the following outlines the computation.

We adopt the conventions
\[
\wt t_j := t_j - t_{j-1}, \qquad \wt b_j := b_j - b_{j-1}, \qquad \wt c_j := c_j - c_{j-1}, \qquad j = 1, \dots, m,
\]
with $t_0 = 0$, $t_m = 1$, and $b_0 = c_0 = b_m = c_m = 0$. We also have $u_1 = v_1$ enforced by our $\delta$-function, compared with the convention of \cite{liu2022} where $u_1 = -v_1$. Define
\[
\Sf g(u; a, b) := e^{a u^2 + b u},
\]
so that $\Sf g(u; a, b) = f(u; 2a, b)$ in the notation (3.4) of \cite{liu2022}, where $f(u; a, b) = e^{\frac{1}{2} a u^2 + b u}$. Then, the identity (3.5) of \cite{liu2022} in our notation is
\begin{equation}\label{eq:gaussid}
\int_{\Gamma} \Sf g(u; a, b)\frac{\dd u}{2\pi i} = \phi_{2a}(b), \qquad a > 0,
\end{equation}
for any contour $\Gamma$ parallel to the $y$-axis with upward orientation, where $\phi_a$ denotes the density of a mean 0 Gaussian random variable with variance $a$. We will use the well-known semigroup property:
\begin{equation}\label{eq:semigroup}
\phi_{a + a'}(b - c) = \int_{\R} \phi_{a}(b - z) \phi_{a'}(c - z)\, \dd z.
\end{equation}
Lastly, let $\Sf q^{\mathrm{br}, z}_{t_1, \dots, t_{m-1}}$ denote the conditional joint probability density function of $\sqrt 2\, \mathbb B^{\mathrm{br}, Z/\sqrt 2}$ given $Z = z$, where
\begin{equation}
\label{eq:condbridgeexpansion}
\mathbb B^{\mathrm{br}, \pm Z/\sqrt 2}(t) := \mathbb B^{\mathrm{br}}(t) \pm \frac{1 - t}{\sqrt 2}\, Z
\end{equation}
for $Z \sim \mathcal N(0,1)$ independent of $\mathbb B^{\mathrm{br}}$. The finite-dimensional density can be written as \begin{equation}\label{eq:bridgedensity}
\phi_{2 t_1}(b_1 - z) \prod_{j=2}^{m} \phi_{2 \wt t_j}(\wt b_j) = \phi_2(z) \Sf q^{\mathrm{br}, z}_{t_1, \dots, t_{m-1}}(b_1, \dots, b_{m-1}).
\end{equation}
We now trace a version of the argument in \cite{liu2022} from (4.3) through to the end of \cite{liu2022}. From \eqref{eq:covintegral},
\begin{multline*}
- \frac{\sqrt{8\pi}}{(\sigma^2L)^{m-1}} \int_{1 + i\R } \frac{\dd v_1}{2\pi i}\int_{{\vec{\Sigma}^{-}}*} \frac{\dd\mb u^*}{(2\pi i )^{m-1}} \int_{{\vec{\Sigma}^{+}*}}  \frac{\dd\mb v^*}{(2\pi i )^{m-1}} \\
\times \lt[\prod_{i=2}^{m-1}\frac{1}{(u_i-u_{i+1})(v_{i+1}-v_i)}\rt]\lt[\frac{\exp(2t_1v_1^2 + 2h_1 v_1)}{(u_1-u_2)(v_2-v_{1})}\rt] \Sf F^*(\mb u^*, \mb v^*) 
\end{multline*}
\begin{align*}
&=\sqrt{8\pi} \int_{1 + i\R} \frac{\dd v_1}{2\pi i} \int_{\vec\Sigma^{-}{}^*} \frac{\dd\mb u^*}{(2\pi i)^{m-1}} \int_{\vec\Sigma^{+}{}^*} \frac{\dd\mb v^*}{(2\pi i)^{m-1}} \int_{\substack{b_j \ge h_j - x_j, \\ c_j \le -h_j - x_j, \\ j = 1, \dots, m-1}} \dd\mb b \dd\mb c \prod_{j=1}^{m} \Sf g(u_j; \wt t_j, \wt b_j) \Sf g(v_j; \wt t_j, -\wt c_j) \\
&\stackrel{\text{Tonelli}}{=} \sqrt{8\pi} \int_{\substack{b_j \ge h_j - x_j, \\ c_j \le -h_j - x_j, \\ j = 1, \dots, m-1}} \dd\mb b \dd\mb c \int_{1 + i\R} \frac{\dd v_1}{2\pi i} \int_{\vec\Sigma^{-}{}^*} \frac{\dd\mb u^*}{(2\pi i)^{m-1}} \int_{\vec\Sigma^{+}{}^*} \frac{\dd\mb v^*}{(2\pi i)^{m-1}} \prod_{j=1}^{m} \Sf g(u_j; \wt t_j, \wt b_j) \Sf g(v_j; \wt t_j, -\wt c_j)  \\
&= \sqrt{8\pi} \int_{\substack{b_j \ge h_j - x_j, \\ c_j \le -h_j - x_j, \\ j = 1, \dots, m-1}} \dd\mb b \dd\mb c \phi_{4 t_1}(b_1 - c_1) \prod_{j=2}^{m} \phi_{2 \wt t_j}(\wt b_j) \phi_{2 \wt t_j}(\wt c_j) \tag{via \eqref{eq:gaussid}}\\
&= \sqrt{8\pi} \int_{\R} \dd z \int_{\substack{b_j \ge h_j - x_j, \\ c_j \le -h_j - x_j, \\ j = 1, \dots, m-1}} \dd\mb b \dd\mb c \lt[ \phi_{2 t_1}(b_1 - z) \prod_{j=2}^{m} \phi_{2 \wt t_j}(\wt b_j) \rt] \lt[ \phi_{2 t_1}(c_1 - z) \prod_{j=2}^{m} \phi_{2 \wt t_j}(\wt c_j) \rt] \; \tag{via  \eqref{eq:semigroup}}
\\
&= \sqrt{8\pi} \int_{\R} \phi_2(z)^2 \lt[ \int_{\substack{b_j \ge h_j - x_j, \\ j = 1, \dots, m-1}} \Sf q^{\mathrm{br}, z}(\mb b) \dd\mb b \rt] \lt[ \int_{\substack{c_j \le -h_j - x_j, \\ j = 1, \dots, m-1}} \Sf q^{\mathrm{br}, z}(\mb c) \dd\mb c \rt] \dd z  \tag{via \eqref{eq:bridgedensity}}
\end{align*}
Note that \[\phi_2(z)^2 = \frac{1}{4\pi} e^{-z^2/2} = \frac{1}{\sqrt{8\pi}} \phi_1(z),\]
so we now have that \begin{multline*}
\lim_{L \to \infty} \frac{\sqrt{8\pi}}{(\sigma^2 L)^{m-1} \Sf Z_L} \int_{\vec\Sigma^-} \frac{\dd\mb u}{(2\pi i)^m} \int_{\vec\Sigma^+} \frac{\dd\mb v}{(2\pi i)^m} \frac{\hat\Pi_{\mb 1}(\mb u, \mb v)}{K_1(u_1, v_1)} \hat{\Sf F}_L(\mb u, \mb v)\, \delta(v_1, u_1) \\
= \p\Big( \sqrt 2 \mathbb B^{\mathrm{br}, Z/\sqrt 2}_+(t_j) \ge h_j - x_j, \sqrt 2 \mathbb B^{\mathrm{br}, Z/\sqrt 2}_-(t_j) \le -h_j - x_j, \; j = 1, \dots, m-1 \Big) \\
= \p\Big( \sqrt 2 \mathbb  B^{\mathrm{br}, Z/\sqrt 2}_+(t_j) \ge h_j - x_j, \sqrt 2  \mathbb B^{\mathrm{br}, -Z/\sqrt 2}_-(t_j) \ge h_j + x_j, \; j = 1, \dots, m-1 \Big).
\end{multline*}
 $\mathbb B^{\mathrm{br}}_+$ and $\mathbb B^{\mathrm{br}}_-$ are independent standard Brownian bridges and $Z \sim \mathcal N(0,1)$ is independent of both. Then, both processes are conditionally independent given $Z$. We may rewrite the limiting distribution as \[
\p\lt( \bigcap_{j=1}^{m-1} \Big\{ \min\big\{ \sqrt 2 \mathbb B^{\mathrm{br}, Z/\sqrt 2}_+(t_j) + x_j, \sqrt 2 \mathbb B^{\mathrm{br}, -Z/\sqrt 2}_-(t_j) - x_j \big\} > h_j \Big\} \rt), 
\]
Using the identity $\min (a,b) = \frac{a+b- |a-b|}{2}$ after substituting \eqref{eq:condbridgeexpansion}, we may rewrite once more, yielding \[
\p\lt( \bigcap_{j=1}^{m-1} \lt\{\lt(\frac{\mathbb B_+^{\mathrm{br}}(t_j) + \mathbb B_-^{\mathrm{br}}(t_j)}{\sqrt{2}}\rt) - \lt|\frac{\mathbb B_+^{\mathrm{br}}(t_j) - \mathbb B_-^{\mathrm{br}}(t_j)}{\sqrt{2}} + (1-t_j)Z + x_j\rt| \rt\} > h_j\rt).
\]
By defining \[
\mathbb B_1^{\mathrm{br}}(t) = \frac{\mathbb B_+^{\mathrm{br}}(t) + \mathbb B_-^{\mathrm{br}}(t)}{\sqrt{2}}, \quad \mathbb B_2^{\mathrm{br}}(t) = \frac{\mathbb B_+^{\mathrm{br}}(t) - \mathbb B_-^{\mathrm{br}}(t)}{\sqrt{2}},
\]
we obtain \[
\p\lt( \bigcap_{j=1}^{m-1} \lt\{ \mathbb B_1^{\mathrm{br}}(t_j)- \lt|\mathbb B_2^{\mathrm {br}}(t_j) + (1-t_j)Z + x_j\rt| \rt\} > h_j\rt) = \Sf P^{\textup{flat}}(\mb t, \mb x, \mb h).
\]
Thus, we have the desired result. \[
\lim_{L \to \infty} \frac{\sqrt{8\pi} \sigma L^{1/2}}{\Sf Z_L} \Sf Q^{(\mb 1), \textup{flat}}_m (\mb M_L, \mb N_L, \mb T_L) = \Sf P^{\textup{flat}}(\mb t, \mb x, \mb h).
\]

\subsection{Proof of Lemma \ref{lem:restofterms}}
\label{sec:proofrestofterms}

As in (6.24) of \cite{BCT}, the $z_i$-contours in the formula of $\Sf Q_m^{(\mb n), \textup{flat}}$ may be taken to be circles of radius 2 centered at the origin. As the $z_i$-dependencies of the flat case are identical to that of the step case, we obtain
\begin{equation}\label{eq:zbound}
\big| \Sf Q^{(\mb n), \textup{flat}}_m(\mb M_L, \mb N_L, \mb T_L) \big| \le 3^{|\mb n|} \max_{|z_i| = 2,\; i = 1, \dots, m-1} \big| \Sf D^{(\mb n), \textup{flat}}_{\mb M_L, \mb N_L, \mb T_L}(\mb z) \big|.
\end{equation}
The cancellation in \eqref{eq:app-nc-flat-cancel}, with the terminal
factor $S_{n_m}$ included, gives
\begin{equation}\label{eq:cancel}
\dot\Pi_{\mb n}(\bs\xi, \bs\eta) := (-1)^{n_1} \frac{\Pi_{\mb n}(\bs\xi, \bs\eta)}{K_{n_1}(\bs\xi^1 | \bs\eta^1)} =  \lt[ \prod_{i=1}^{m-1} K_{n_i + n_{i+1}}(\bs\xi^i, \bs\eta^{i+1} | \bs\eta^i, \bs\xi^{i+1}) \rt] K_{n_m}(\bs\xi^m | \bs\eta^m) S_{n_m}(\bs\xi^m | \bs\eta^m).
\end{equation}
Consequently, we may write the integrand of $\Sf D^{(\mb n), \textup{flat}}$ as $\dot \Pi_{\mb n}(\bs\xi, \bs\eta)\, \Sf F^{(\mb n)}_{\mb M_L, \mb N_L, \mb T_L}(\bs\xi, \bs\eta) \det[\delta(-\eta^1_r - 1, \xi^1_s)]_{r,s=1}^{n_1}$,

Recall the circles $\Sigma^{b,L}_\pm$ of Lemma 5.3 of \cite{BCT}. We take 
\[
C_{1,R} = \Sigma^{0,L}_+, \qquad C_{1,L} = -1 - C_{1,R} = \Sigma^{0,L}_-,
\]
for $i = 2, \dots, m$, as in Subsection 6.3 of \cite{BCT}, we take
\[
C^{\text{in}}_{i,L} = \Sigma^{-(i-1), L}_-, \quad C^{\text{out}}_{i,L} = \Sigma^{i-1, L}_-, \quad C^{\text{in}}_{i,R} = \Sigma^{-(i-1), L}_+, \quad C^{\text{out}}_{i,R} = \Sigma^{i-1, L}_+.
\]
For all sufficiently large $L$, all contours are contained in the disk $\{|z| \le 2\}$, the distance between any two distinct contours on the same side is at least $L^{-1/2}$, and the distance between any contour on the left and any contour on the right is bounded below by a positive constant. Expanding the determinant,
\[
\det\big[ \delta(-\eta^1_r - 1, \xi^1_s) \big]_{r,s=1}^{n_1} = \sum_{\pi \in S_{n_1}} \operatorname{sgn}(\pi) \prod_{s=1}^{n_1} \delta(-\eta^1_{\pi(s)} - 1, \xi^1_s),
\]
and we perform each of the $n_1!$ integrals, where we are coupling \[
\xi^1_s = -\eta^1_{\pi(s)} - 1, \qquad s = 1, \dots, n_1.
\]
We may first bound the integrals in the expression of $\Sf D^{(\mb n), \textup{flat}}_{\mb M, \mb N, \mb L}$ by mass and length, then multiply by $n_1!$, so we obtain the following bound akin to (6.25) in \cite{BCT}, \begin{equation}\label{eq:Dbound}
\big| \Sf D^{(\mb n), \textup{flat}}_{\mb M_L, \mb N_L, \mb T_L}(\mb z) \big| \le C^{|\mb n|} n_1! \max \big| \dot\Pi_{\mb n}(\bs\xi, \bs\eta) \Sf F^{(\mb n)}_{\mb M_L, \mb N_L, \mb T_L}(\bs\xi, \bs\eta) \big|,
\end{equation}
where $C$ is some sufficiently large constant and the maximum is taken over $\bs\eta^1 \in (C_{1,R})^{n_1}$, over $\bs\xi^i, \bs\eta^i$ on the unions $C^{\text{in}} \cup C^{\text{out}}$ of the corresponding contours for $i = 2, \dots, m$, and with $\bs\xi^1$ given by the substitution above enforced by the $\delta$-kernel.

We next bound the two factors $\dot\Pi_{\mb n}(\bs\xi, \bs\eta)$, $ \Sf F^{(\mb n)}_{\mb M_L, \mb N_L, \mb T_L}(\bs\xi, \bs\eta)$ appearing on the right-hand side of
\eqref{eq:Dbound}. 

Recall that, for two vectors
$\mb w=(w_1,\ldots,w_q)$ and $\mb w'=(w'_1,\ldots,w'_q)$ satisfying
    $\min_{r,s}|w_r-w'_s|\ge d>0$,
Hadamard's inequality gives
\begin{equation}
\label{eq:hadamard-cauchy-flat}
    \left|K_q(w\mid w')\right|
    =
    \left|
        \det\left[\frac{1}{w_r-w'_s}\right]_{r,s=1}^{q}
    \right|
    \le
    \prod_{r=1}^{q}
    \left(
        \sum_{s=1}^{q}\frac{1}{|w_r-w'_s|^2}
    \right)^{1/2}
    \le
    \frac{q^{q/2}}{d^q}.
\end{equation}

After imposing
\[
    \xi_s^1=-1-\eta_{\pi(s)}^1,
    \qquad s=1,\ldots,n_1,
\]
the variables $\xi_s^1$ lie on $C_{1,L}=\Sigma_-^{0,L}$.
Thus, every pair of contours occurring in a Cauchy determinant in
$\dot\Pi_{\mathbf n}$ is either separated by a positive constant or
by at least $L^{-1/2}$. Notice that the delta coupling does not create a
collision: it maps $C_{1,R}=\Sigma_+^{0,L}$ exactly onto
$C_{1,L}=\Sigma_-^{0,L}$.

Applying \eqref{eq:hadamard-cauchy-flat} to all the Cauchy
determinants in $\dot\Pi_{\mathbf n}$ therefore yields
\begin{align}
    \left|\dot\Pi_{\mathbf n}(\boldsymbol\xi,\boldsymbol\eta)\right|
    &\le
    C^{|\mathbf n|}
    L^{|\mathbf n|-n_1/2}
    \left[
        \prod_{i=1}^{m-1}
        (n_i+n_{i+1})^{(n_i+n_{i+1})/2}
    \right]
    n_m^{n_m/2}
    \left|S_{n_m}(\xi^m\mid\eta^m)\right|.
\label{eq:Pi-dot-prebound}
\end{align}
Here, the power $L^{|\mathbf n|-n_1/2}$, rather than
$L^{|\mathbf n|}$, reflects the absence of the initial factor
$K_{n_1}(\eta^1\mid\xi^1)$.

Since all contours lie in $\{z \in \mathbb C: |z|\le 2\}$,
\[
    \left|S_{n_m}(\xi^m\mid\eta^m)\right|
    =
    \left|
        \sum_{k=1}^{n_m}(\xi_k^m-\eta_k^m)
    \right|
    \le 4n_m.
\]
Using
\[
    q^q\le 4^q q!,
    \qquad
    (p+q)^{p+q}\le 8^{p+q}p!q!,
\]
and enlarging $C$ if necessary, we conclude from
\eqref{eq:Pi-dot-prebound} that
\begin{equation}
\label{eq:Pi-dot-bound}
    \max
    \left|\dot\Pi_{\mathbf n}
        (\boldsymbol\xi,\boldsymbol\eta)\right|
    \le
    C^{|\mathbf n|}
    \mathbf n!\,
    L^{|\mathbf n|-n_1/2}.
\end{equation}

It remains to bound the exponential factor $ \Sf F^{(\mb n)}_{\mb M_L, \mb N_L, \mb T_L}(\bs\xi, \bs\eta)$. The delta substitution
does not alter the steepest-descent estimate. Consequently, Lemma 5.3 of \cite{BCT} and
\eqref{eq:frati} imply that, uniformly over all the contours,
\begin{align}
    \frac{1}{\Sf Z_L}
    \left|
        \Sf F_{\mathbf M_L,\mathbf N_L,\mathbf T_L}^{(\mathbf n)}
        (\boldsymbol\xi,\boldsymbol\eta)
    \right|
    &\le
    C^{2|\mathbf n|}
    \prod_{i=1}^{m}
    \left|
        \frac{\Sf f_{L,i}(z^-)}{\Sf f_{L,i}(z^+)}
    \right|^{n_i-1} \notag\\
    &\le
    C^{2|\mathbf n|}
    \exp\left\{
        -\frac{\tau J(\ell)}{2}
        (|\mathbf n|-m)L
    \right\}.
\label{eq:F-flat-uniform-bound}
\end{align}

Combining \eqref{eq:zbound}, \eqref{eq:Dbound}, \eqref{eq:Pi-dot-bound},
and \eqref{eq:F-flat-uniform-bound},
we obtain
\begin{equation}
\label{eq:Q-flat-n-bound}
    \frac{1}{\Sf Z_L}
    \left|
       \Sf Q_m^{(\mathbf n),\mathrm{flat}}
        (\mathbf M_L,\mathbf N_L,\mathbf T_L)
    \right|
    \le
    C^{|\mathbf n|}
    n_1!\,\mathbf n!\,
    L^{|\mathbf n|-n_1/2}
    \exp\left\{
        -\frac{\tau J(\ell)}{2}
        (|\mathbf n|-m)L
    \right\}.
\end{equation}
The factor $n_1!$ comes from the $n_1!$ terms in the expansion of
the delta determinant.

If $\mathbf n\in \mathbb N^m \setminus\{\mathbf 1\}$, then $|\mathbf n|\ge m+1$, and hence
    $|\mathbf n|-m
    \ge \frac{|\mathbf n|}{m+1}$. Since $m$ is fixed and $J(\ell)>0$, the polynomial and constant
factors in \eqref{eq:Q-flat-n-bound} can be absorbed into the
exponential. Thus, if $\mathbf n\in \mathbb N^m \setminus\{\mathbf 1\}$, there exist constants $c_1>0$ and $L_1>0$ such
that
\begin{equation}
\label{eq:Q-flat-exp-bound}
    \frac{1}{\Sf Z_L}
    \left|
        \Sf Q_m^{(\mathbf n),\mathrm{flat}}
        (\mathbf M_L,\mathbf N_L,\mathbf T_L)
    \right|
    \le
    e^{-c_1|\mathbf n|L}n_1!\,\mathbf n!,
\end{equation}
for every $L\ge L_1$.

Set $q_L=e^{-c_1L}$. We find
\begin{align*}
    \frac{1}{\Sf Z_L}
    \sum_{\mathbf n\in\mathbb N^m\setminus\{\mathbf 1\}}
    \frac{
        \left|
            \Sf Q_m^{(\mathbf n),\mathrm{flat}}
            (\mathbf M_L,\mathbf N_L,\mathbf T_L)
        \right|
    }{(\mathbf n!)^2}
    &\le
    \sum_{\mathbf n\in\mathbb N^m}
    \frac{q_L^{|\mathbf n|}}{n_2!\cdots n_m!} \\
    &=
    \left(\sum_{n_1=1}^{\infty}q_L^{n_1}\right)
    \prod_{i=2}^{m}
    \left(\sum_{n_i=1}^{\infty}
        \frac{q_L^{n_i}}{n_i!}\right) \\
    &=
    \frac{q_L}{1-q_L}
    (e^{q_L}-1)^{m-1}.
\end{align*}
Therefore, there exist $c>0$ and
$L_2>0$ such that
\[
    \frac{1}{\Sf Z_L}
    \sum_{\mathbf n\in\mathbb N^m\setminus\{\mathbf 1\}}
    \frac{
        \left|
            \Sf Q_m^{(\mathbf n),\mathrm{flat}}
            (\mathbf M_L,\mathbf N_L,\mathbf T_L)
        \right|
    }{(\mathbf n!)^2}
    \le e^{-cL}
\]
for every $L\ge L_2$. This proves Lemma \ref{lem:restofterms}.

\section{Conditional Fluctuations Heuristics}
\label{sec:flat-conditional-fluctuation-heuristic}

We follow the reasoning of \cite[Section~2.2]{BCT}, but include the
additional Gaussian component arising in the flat model. At the level of
the conditional law of large numbers, the candidate mid-way points on
the favorable path are
\[
    (atN,atN),\qquad t\in(0,1).
\]
To capture the next-order fluctuations, we enlarge this family by
allowing an \(N^{1/2}\)-scale transverse displacement. Thus, for
\((t,s)\in(0,1)\times\mathbb R\), we consider the more general mid-way
points
\begin{equation}
    \mathbf{p}_t^sN
    =
    \left(
        atN+s\frac{a\sigma}{\sqrt D}N^{1/2},
        atN-s\frac{a\sigma}{\sqrt D}N^{1/2}
    \right).
    \label{eq:flat-fluctuation-midpoint}
\end{equation}
The parameter \(t\) specifies the macroscopic attachment location on
the conditioned backbone, while \(s\) describes its transverse
displacement on the fluctuation scale.

Theorem~\ref{thm:cond-diag-fluc} gives the \(N^{1/2}\)-scale fluctuations of
\(\mathcal L^{\mathrm{flat}}(\mathbf{p}_t^sN)\) along this enlarged family of
mid-way points. In the notation of that theorem, its contribution is
\begin{equation}
    t\ell N
    +
    \sigma 
    \left(
        \B_1(t)
        -
        \left|
            \B_2(t)-(1-t)Z-s
        \right|
    \right)N^{1/2}.
    \label{eq:flat-diagonal-fluctuation-input}
\end{equation}

Suppose that the maximizing path to \((xN,yN)\) uses the conditioned
backbone. The conditional-LLN calculation shows that its macroscopic
attachment time is \(t=t_c\). Following the favorable path up to a
general mid-way point \(\mathbf{p}_{t_c}^sN\), and then continuing along an
ordinary point-to-point geodesic to \((xN,yN)\), leads to the
variational approximation
\begin{equation}
    \mathcal L^{\mathrm{flat}}(xN,yN)
    \approx
    \max_{s\in\mathbb R}
    \left\{
        \mathcal L^{\mathrm{flat}}\bigl(\mathbf{p}_{t_c}^s N\bigr)
        +
        \mathcal L_{\mathbf{p}_{t_c}^sN}(xN,yN)
    \right\}.
    \label{eq:flat-midway-variational-ansatz}
\end{equation}
Here \(\mathcal L_{\mathbf{p}_{t_c}^sN}(xN,yN)\) denotes the ordinary
point-to-point last passage time from \(\mathbf{p}_{t_c}^sN\) to
\((xN,yN)\). Thus, after fixing the macroscopic optimizer \(t_c\), it
remains to optimize over the transverse displacement \(s\).

Using unconditional result, we expect
\begin{equation}
\begin{aligned}
\mathcal L_{\mathbf p_{t_c}^sN}(xN,yN)
\approx\;&
\left(
  \sqrt{x-at_c-s\frac{a\sigma}{\sqrt D}N^{-1/2}}
  +
  \sqrt{y-at_c+s\frac{a\sigma}{\sqrt D}N^{-1/2}}
\right)^2N \\
&+
\frac{
  \left(\sqrt{x-at_c}+\sqrt{y-at_c}\right)^{4/3}
}{
  \left((x-at_c)(y-at_c)\right)^{1/6}
}
\TW_2N^{1/3} \\
=\;&
\left(\sqrt{x-at_c}+\sqrt{y-at_c}\right)^2N
+\sigma sR(t_c)N^{1/2} \\
&+
\frac{
  \left(\sqrt{x-at_c}+\sqrt{y-at_c}\right)^{4/3}
}{
  \left((x-at_c)(y-at_c)\right)^{1/6}
}
\TW_2N^{1/3}
+O(1).
\end{aligned}
\label{eq:flat-continuation-shape-expansion}
\end{equation}
where
\begin{equation}
    R(t)
    =
    \frac{a}{\sqrt D}
    \left(
        \sqrt{\frac{x-at}{y-at}}
        -
        \sqrt{\frac{y-at}{x-at}}
    \right).
    \label{eq:flat-continuation-linear-coefficient}
\end{equation}
Combining \eqref{eq:flat-diagonal-fluctuation-input} and
\eqref{eq:flat-continuation-shape-expansion}, and setting \(t=t_c\),
the \(N^{1/2}\)-order contribution of the conditioned route is
\begin{equation}
    \sigma\mathcal ZN^{1/2},
    \qquad
    \mathcal Z
    =
    \max_{s\in\mathbb R}
    \left\{
        \B_1(t_c)
        -
        \left|
            \B_2(t_c)-(1-t_c)Z-s
        \right|
        +
        R(t_c)s
    \right\}.
    \label{eq:flat-maximized-second-order-contribution}
\end{equation}

\begin{itemize}
    \item If \((x,y)\in\mathcal R_1^\circ\), then \(t_c=1\). The Brownian
bridges vanish at \(1\), the Gaussian term is multiplied by
\(1-t_c=0\), and the strict inequalities defining the interior of
\(\Omega_1\) give \(\lvert R(1)\rvert<1\). Hence the maximum in
\eqref{eq:flat-maximized-second-order-contribution} is attained at
\(s=0\) and equals zero. The \(N^{1/2}\)-term therefore cancels,
leaving the ordinary point-to-point fluctuation from \((aN,aN)\) to
\((xN,yN)\). Thus, using unconditional point-to-point fluctuations result \eqref{eq:uncond-ptp-fluc},
\[
    \mathcal L^{\mathrm{flat}}(xN,yN)
    \approx
    h^{\mathrm{flat}}(x,y)N + 
        \frac{(\sqrt{x-a}+\sqrt{y-a})^{4/3}}{((x-a)(y-a))^{1/6}}
        \TW_2N^{1/3},
\]
which yields \eqref{eq:flat-end-fluctuation-conjecture}.\

\item If  \((x,y)\in\mathcal R_2^\circ\) and \(x\neq y\). If \(y<x\),
then
    $\frac{y-at_c}{x-at_c}=\frac1m,
    R(t_c)=1$; if \(y>x\), then
$
    \frac{y-at_c}{x-at_c}=m,
    R(t_c)=-1.
$ Consequently,
$\mathcal Z
    =
    \B_1(t_c)\pm\B_2(t_c)\mp(1-t_c)Z
    =
    \sqrt2\,\B_\pm(t_c)\mp(1-t_c)Z
$ for $y<x$ and $y>x$ separately. The point-to-point \(N^{1/3}\)-fluctuation of the continuation is
negligible on this \(N^{1/2}\) scale. Therefore, 
\[
    \mathcal L^{\mathrm{flat}}(xN,yN)
    \approx
    h^{\mathrm{flat}}(x,y)N + 
        \bigl(\sqrt2\,\B_\pm(t_c)\mp(1-t_c)Z\bigr)N^{1/2}
\]for $y<x$ and $y>x$ separately, giving
\eqref{eq:flat-middle-fluctuation-conjecture}.
\item Outside \(\overline{\mathcal R_1\cup\mathcal R_2}\), the
ordinary flat orthogonal route has the strictly larger first-order value, and the
conditioning should not affect its leading fluctuation. Translating the
endpoint along the anti-diagonal reduces the latter problem to the
diagonal point
\[
    \left(\frac{x+y}{2}N,\frac{x+y}{2}N\right).
\]
Therefore, by the unconditional line-to-point fluctuations result \eqref{eq:uncond-ltp-fluc}, we have
\[
     \mathcal L^{\mathrm{flat}}(xN,yN)
    \approx
    h^{\mathrm{flat}}(x,y)N + 2^{1/3}(x+y)^{1/3}
        \TW_1N^{1/3},
\]
which gives \eqref{eq:flat-unaffected-fluctuation-conjecture}.
\end{itemize}
This completes the heuristic argument for Conjecture \ref{conj:flat-fluctuations}.

\newpage

\appendix

\section{Notation correspondence for the line-to-point formula}
\label{app:notation-flat}

We derive the notation in Section~\ref{sec:multi-point-formula} from
Liu's multi-point TASEP formula \cite{liu2021}, using the Cauchy-determinant
convention of \cite{BCT}. The calculation separates three ingredients:
the exponential weights, the rational factors in the integration variables,
and the auxiliary contour parameters. In particular, it identifies the
additional factor associated with the flat initial condition and explains
the passage from the mixed distribution to its conditional density.

\subsection{Parameters and exponential weights}
\label{app:nc-parameters}

Assume the hypotheses of Proposition~\ref{prop:multipointformula}.
Write $x_k(t)$ for TASEP with initial positions $x_k(0)=-2k$,
$k\in\Z$. Under the usual coupling with the line-to-point model of
Definition~\ref{def:flat-explpp}, whose weights on the starting line
$\mathbb L$ are zero, we have
\begin{equation}
\label{eq:app-nc-tasep}
    \{\mathcal L^{\mathrm{flat}}(M_i,N_i)>T_i\}
    =\{x_{N_i}(T_i)<M_i-N_i\}.
\end{equation}
Thus Liu's parameters are $a_i=M_i-N_i$, $k_i=N_i$, and $t_i=T_i$.
To distinguish the auxiliary variables before and after the change of
sign, we denote Liu's auxiliary variables by
$\bs\zeta=(\zeta_1,\ldots,\zeta_{m-1})$ and set $\bs\zeta=-\mb z$.
The remaining correspondences use the
\hyperref[not:integration-variables]{integration variables} and
\hyperref[not:multipoint-contours]{nested contours} introduced in
Section~\ref{sec:multipoint-notation}:
\[
\begin{array}{c|c}
\text{Notation in \cite{liu2021}} & \text{Notation used here} \\ \hline
U^{(i)}=(u_1^{(i)},\ldots,u_{n_i}^{(i)})
    & \bs\xi^i=(\xi_1^i,\ldots,\xi_{n_i}^i) \\
V^{(i)}=(v_1^{(i)},\ldots,v_{n_i}^{(i)})
    & \bs\eta^i=(\eta_1^i,\ldots,\eta_{n_i}^i) \\
\Sigma_{1,L},\ \Sigma_{1,R} & C_{1,L},\ C_{1,R} \\
\Sigma_{i,L}^{\mathrm{in/out}},\ \Sigma_{i,R}^{\mathrm{in/out}}
    & C_{i,L}^{\mathrm{in/out}},\ C_{i,R}^{\mathrm{in/out}}\quad(i\geq2).
\end{array}
\]
The contours retain the nesting and counterclockwise orientations
specified there. For the flat case, choose in addition
$C_{1,L}=-1-C_{1,R}$, with the left and right contours contained in
$\{\operatorname{Re}w<-1/2\}$ and $\{\operatorname{Re}w>-1/2\}$,
respectively.

Liu's function $F_i$ is the
\hyperref[not:exponential-weight]{exponential weight $f_{M_i,N_i,T_i}$}
defined in the main text:
\[
    F_i(w)=w^{k_i}(1+w)^{-a_i-k_i}e^{t_iw}
           =f_{M_i,N_i,T_i}(w),\qquad F_0(w)=1.
\]
To distinguish Liu's $f_i$ from $\Sf f_i$ in
\eqref{eq:weight-product-main}, denote the former by $f_i^{\mathrm{Liu}}$.
Then
\begin{equation}
\label{eq:app-nc-weights}
    f_i^{\mathrm{Liu}}(w)=
    \begin{cases}
        \Sf f_i(w),&w\in\Omega_L\setminus\{-1\},\\
        \Sf f_i(w)^{-1},&w\in\Omega_R\setminus\{0\},
    \end{cases}
    \qquad
    \Sf f_i(w)=\frac{f_{M_i,N_i,T_i}(w)}
                       {f_{M_{i-1},N_{i-1},T_{i-1}}(w)}.
\end{equation}
Here $M_0=N_0=T_0=0$. Consequently, the product of exponential
weights in Liu's integrand is exactly
$\Sf F^{(\mb n)}_{\mb M,\mb N,\mb T}$ from
\eqref{eq:weight-product-main}:
\begin{equation}
\label{eq:app-nc-F}
    \prod_{i=1}^m\prod_{j=1}^{n_i}
       f_i^{\mathrm{Liu}}(\xi_j^i)f_i^{\mathrm{Liu}}(\eta_j^i)
    =\prod_{i=1}^m\prod_{j=1}^{n_i}
       \frac{\Sf f_i(\xi_j^i)}{\Sf f_i(\eta_j^i)}
    =\Sf F^{(\mb n)}_{\mb M,\mb N,\mb T}(\bs\xi,\bs\eta).
\end{equation}

Throughout this appendix, $\mb n\in\N_0^m$,
$|\mb n|=\sum_i n_i$, and $\mb n!=\prod_i n_i!$.
Empty products and determinants are $1$; in particular, $K_0=1$,
whereas $S_0=0$.

\subsection{Vandermonde products and Cauchy determinants}
\label{app:nc-cauchy}

Use Liu's Vandermonde convention
\[
    \Delta(\mb r)=\prod_{i<j}(r_j-r_i),\qquad
    \Delta(\mb r;\mb s)=\prod_{i,j}(r_i-s_j).
\]
For vectors $\mb r,\mb s\in\C^n$, the
\hyperref[not:cauchy-determinant]{Cauchy determinant $K_n$}
defined in the main text satisfies
\begin{equation}
\label{eq:app-nc-cauchy-sign}
    K_n(\mb r\mid\mb s)
      =(-1)^{n(n-1)/2}
        \frac{\Delta(\mb r)\Delta(\mb s)}{\Delta(\mb r;\mb s)},
    \qquad
    K_n(\mb s\mid\mb r)=(-1)^nK_n(\mb r\mid\mb s).
\end{equation}
In particular, the sign in the reciprocal identity is also
$(-1)^{n(n-1)/2}$:
\[
    \frac{\Delta(\mb r;\mb s)}{\Delta(\mb r)\Delta(\mb s)}
      =\frac{(-1)^{n(n-1)/2}}{K_n(\mb r\mid\mb s)}.
\]

For $1\leq i<m$, set
\[
    C_i(\bs\xi,\bs\eta)
      :=\frac{\Delta(\bs\xi^i;\bs\eta^{i+1})
                    \Delta(\bs\eta^i;\bs\xi^{i+1})}
                   {\Delta(\bs\xi^i;\bs\xi^{i+1})
                    \Delta(\bs\eta^i;\bs\eta^{i+1})}.
\]
Factoring the Cauchy determinant for the two concatenated blocks gives
\begin{equation}
\label{eq:app-nc-block}
\begin{split}
    &K_{n_i+n_{i+1}}
       (\bs\xi^i,\bs\eta^{i+1}\mid\bs\eta^i,\bs\xi^{i+1})\\
    &\hspace{1cm}=
       K_{n_i}(\bs\xi^i\mid\bs\eta^i)
       K_{n_{i+1}}(\bs\eta^{i+1}\mid\bs\xi^{i+1})
       C_i(\bs\xi,\bs\eta).
\end{split}
\end{equation}
The undifferentiated factor
\hyperref[not:undifferentiated-flat]{$\wt\Pi_{\mb n}$}
used in the proof of Proposition~\ref{prop:multipointformula} is
\begin{equation}
\label{eq:app-nc-Pitilde}
\begin{split}
    \wt\Pi_{\mb n}(\bs\xi,\bs\eta)
      :={}&K_{n_1}(\bs\eta^1\mid\bs\xi^1)\\
      &\times\left[\prod_{i=1}^{m-1}
       K_{n_i+n_{i+1}}
       (\bs\xi^i,\bs\eta^{i+1}\mid\bs\eta^i,\bs\xi^{i+1})\right]
       K_{n_m}(\bs\xi^m\mid\bs\eta^m).
\end{split}
\end{equation}
Thus the rational factor in \eqref{eq:Pi-main} satisfies
$\Pi_{\mb n}=\wt\Pi_{\mb n}S_{n_m}(\bs\xi^m\mid\bs\eta^m)$.
Multiplying \eqref{eq:app-nc-block} over $i$ and then attaching the two
end factors yields
\begin{equation}
\label{eq:app-nc-vandermonde}
    \left[\prod_{i=1}^m
      \frac{\Delta(\bs\xi^i)^2\Delta(\bs\eta^i)^2}
           {\Delta(\bs\xi^i;\bs\eta^i)^2}\right]
    \prod_{i=1}^{m-1}C_i(\bs\xi,\bs\eta)
    =(-1)^{|\mb n|}\wt\Pi_{\mb n}(\bs\xi,\bs\eta).
\end{equation}
Indeed, the product on the right of \eqref{eq:app-nc-Pitilde} reduces to
$\prod_i K_{n_i}(\bs\xi^i\mid\bs\eta^i)
K_{n_i}(\bs\eta^i\mid\bs\xi^i)\prod_i C_i$,
and \eqref{eq:app-nc-cauchy-sign} supplies the sign $(-1)^{|\mb n|}$.

\subsection{The flat initial-condition factor}
\label{app:nc-flat-factor}

The additional factor in \eqref{eq:Dnflat}, relative to the step
case, comes from Liu's flat series \cite[Eq.~(2.19)]{liu2021}:
\begin{equation}
\label{eq:app-nc-flat-factor}
\begin{split}
    B_{n_1}^{\mathrm{flat}}(\bs\xi^1,\bs\eta^1)
       &:=(-1)^{n_1(n_1+1)/2}
       \frac{\Delta(\bs\xi^1;\bs\eta^1)}
            {\Delta(\bs\xi^1)\Delta(\bs\eta^1)}
       \det\!\left[\delta(-\eta_r^1-1,\xi_s^1)\right]_{r,s=1}^{n_1}\\
       &=\frac{(-1)^{n_1}}{K_{n_1}(\bs\xi^1\mid\bs\eta^1)}
       \det\!\left[\delta(-\eta_r^1-1,\xi_s^1)\right]_{r,s=1}^{n_1}.
\end{split}
\end{equation}
The last equality follows because
$n_1(n_1+1)/2+n_1(n_1-1)/2=n_1^2$ has the same parity as $n_1$.

With the
\hyperref[not:delta-kernel]{delta-kernel normalization},
we have
\begin{equation}
\label{eq:app-nc-delta}
    \int_{C_{1,L}}\delta(-\eta-1,\xi)g(\xi)
           \frac{\dd\xi}{2\pi i}=g(-\eta-1),
    \qquad \eta\in C_{1,R},
\end{equation}
for continuous test functions $g$ on $C_{1,L}$. The condition
$C_{1,L}=-1-C_{1,R}$ ensures that the evaluation point belongs to the
integration contour. Using the
\hyperref[not:delta-determinant]{determinant convention},
we expand
\[
    \det\!\left[\delta(-\eta_r^1-1,\xi_s^1)\right]_{r,s=1}^{n_1}
     =\sum_{\sigma\in\mathfrak S_{n_1}}\operatorname{sgn}(\sigma)
       \prod_{r=1}^{n_1}\delta(-\eta_r^1-1,\xi_{\sigma(r)}^1)
\]
and applying \eqref{eq:app-nc-delta} in each $\xi^1$ variable.
In particular, this operation introduces no further factor of $2\pi i$.

There is also a useful cancellation:
\begin{equation}
\label{eq:app-nc-flat-cancel}
\begin{split}
    \wt\Pi_{\mb n}B_{n_1}^{\mathrm{flat}}
      ={}&\left[\prod_{i=1}^{m-1}
        K_{n_i+n_{i+1}}
        (\bs\xi^i,\bs\eta^{i+1}\mid\bs\eta^i,\bs\xi^{i+1})\right]\\
      &\times K_{n_m}(\bs\xi^m\mid\bs\eta^m)
       \det\!\left[\delta(-\eta_r^1-1,\xi_s^1)\right]_{r,s=1}^{n_1}.
\end{split}
\end{equation}
Multiplying the Cauchy-determinant product on the right by
$S_{n_m}(\bs\xi^m\mid\bs\eta^m)$ gives the factor $\dot\Pi_{\mb n}$
defined in \eqref{eq:cancel}. Quotients involving $K_{n_1}$
in the flat integrand may always be interpreted through this cancellation.

\subsection{Contour measures and the change of auxiliary variables}
\label{app:nc-measures}

Liu's measure \cite[Eq.~(2.4)]{liu2021}, after $\bs\zeta=-\mb z$,
is, for $i\geq2$ and $\alpha\in\{L,R\}$,
\begin{equation}
\label{eq:app-nc-measure}
\begin{split}
    \int_{\Sigma_{i,\alpha}}g(w)\,\dd\mu_{-\mb z}(w)
       =\frac{1}{1+z_{i-1}}
       \left[\int_{C_{i,\alpha}^{\mathrm{in}}}g(w)\frac{\dd w}{2\pi i}
       +z_{i-1}\int_{C_{i,\alpha}^{\mathrm{out}}}g(w)
                    \frac{\dd w}{2\pi i}\right].
\end{split}
\end{equation}
At level $1$, the measure is simply $\dd w/(2\pi i)$ on $C_{1,\alpha}$.
Thus the coefficient of the inner contour is $1$, and that of the outer
contour is $z_{i-1}$, after extracting $(1+z_{i-1})^{-1}$.

For a function or delta-kernel integrand $G$, group the contour
integrations appearing in \eqref{eq:Dnflat} into the operator
\begin{equation}
\label{eq:app-nc-I}
\begin{split}
    \mathcal I_{\mb n}(\mb z)[G]
    :={}&\prod_{i=2}^m\prod_{j=1}^{n_i}
       \left[\int_{C_{i,L}^{\mathrm{in}}}\frac{\dd\xi_j^i}{2\pi i}
         +z_{i-1}\int_{C_{i,L}^{\mathrm{out}}}
                         \frac{\dd\xi_j^i}{2\pi i}\right]\\
    &\times\prod_{i=2}^m\prod_{j=1}^{n_i}
       \left[\int_{C_{i,R}^{\mathrm{in}}}\frac{\dd\eta_j^i}{2\pi i}
         +z_{i-1}\int_{C_{i,R}^{\mathrm{out}}}
                         \frac{\dd\eta_j^i}{2\pi i}\right]\\
    &\times\prod_{j=1}^{n_1}
       \left[\int_{C_{1,L}}\frac{\dd\xi_j^1}{2\pi i}
             \int_{C_{1,R}}\frac{\dd\eta_j^1}{2\pi i}\right]
       G(\bs\xi,\bs\eta).
\end{split}
\end{equation}
Products of integration operators act on the entire integrand to their
right. In the flat case the delta evaluations are performed first.

The explicit auxiliary factors in Liu's series become
\[
    \prod_{i=1}^{m-1}(1-\zeta_i)^{n_i}
          (1-\zeta_i^{-1})^{n_{i+1}}
      =\prod_{i=1}^{m-1}(1+z_i)^{n_i+n_{i+1}}z_i^{-n_{i+1}}.
\]
Meanwhile, the extracted measure factors give
\[
    \prod_{i=2}^m(1+z_{i-1})^{-2n_i}
       =\prod_{i=1}^{m-1}(1+z_i)^{-2n_{i+1}}.
\]
Combining these identities with \eqref{eq:app-nc-F},
\eqref{eq:app-nc-vandermonde}, and \eqref{eq:app-nc-flat-factor}, we obtain
\begin{equation}
\label{eq:app-nc-D-conversion}
    \mathcal D_{\mb n,Y_{\mathrm{flat}}}(-\mb z)
      =(-1)^{|\mb n|}
       \wt{\Sf D}^{(\mb n),\mathrm{flat}}_{\mb M,\mb N,\mb T}(\mb z)
       \prod_{i=1}^{m-1}
          \frac{(1+z_i)^{n_i-n_{i+1}}}{z_i^{n_{i+1}}},
\end{equation}
where the
\hyperref[not:undifferentiated-flat]{undifferentiated integral} is
\begin{equation}
\label{eq:app-nc-Dtilde}
    \wt{\Sf D}^{(\mb n),\mathrm{flat}}_{\mb M,\mb N,\mb T}(\mb z)
      :=\mathcal I_{\mb n}(\mb z)
       \left[\wt\Pi_{\mb n}\,
       \Sf F^{(\mb n)}_{\mb M,\mb N,\mb T}\,
       B_{n_1}^{\mathrm{flat}}\right].
\end{equation}
Equivalently,
\[
    \wt{\Sf D}^{(\mb n),\mathrm{flat}}_{\mb M,\mb N,\mb T}(\mb z)
      =(-1)^{|\mb n|}\mathcal D_{\mb n,Y_{\mathrm{flat}}}(-\mb z)
       \prod_{i=1}^{m-1}
          \frac{z_i^{n_{i+1}}}{(1+z_i)^{n_i-n_{i+1}}}.
\]
Expanding \eqref{eq:app-nc-Dtilde} gives precisely
\eqref{eq:Dnflat} with $\Pi_{\mb n}$ replaced by $\wt\Pi_{\mb n}$,
including its prefactor $(-1)^{n_1}(2\pi i)^{-2|\mb n|}$.

\subsection{From the mixed distribution to the conditional density}
\label{app:nc-conditioning}

Set
\[
    E_{m-1}^{\mathrm{flat}}
       :=\bigcap_{i=1}^{m-1}
           \{\mathcal L^{\mathrm{flat}}(M_i,N_i)>T_i\},\qquad
    H_m^{\mathrm{flat}}
       :=\p\big(E_{m-1}^{\mathrm{flat}},\,
                    \mathcal L^{\mathrm{flat}}(M_m,N_m)\leq T_m\big).
\]
The mixed-event version of Liu's formula (Proposition~2.3, with the flat
specialization of Proposition~2.16 in \cite{liu2021}) reads
\begin{equation}
\label{eq:app-nc-Liu-probability}
    H_m^{\mathrm{flat}}
      =\frac{(-1)^{m-1}}{(2\pi i)^{m-1}}
       \oint_{>1}\cdots\oint_{>1}
       \mathcal D_{Y_{\mathrm{flat}}}(\bs\zeta)
       \prod_{i=1}^{m-1}\frac{\dd\zeta_i}{\zeta_i(1-\zeta_i)},
\end{equation}
where
\[
    \mathcal D_{Y_{\mathrm{flat}}}(\bs\zeta)
      =\sum_{\mb n\in\N_0^m}\frac{1}{(\mb n!)^2}
          \mathcal D_{\mb n,Y_{\mathrm{flat}}}(\bs\zeta).
\]
Here $\oint_{>1}$ denotes a counterclockwise circle centered at $0$
with radius greater than $1$. The substitution $\zeta_i=-z_i$ preserves
this orientation and gives $\dd\zeta_i/\zeta_i=\dd z_i/z_i$;
there is no additional sign from the differentials. Substituting
\eqref{eq:app-nc-D-conversion} into
\eqref{eq:app-nc-Liu-probability} therefore gives
\begin{equation}
\label{eq:app-nc-mixed}
\begin{split}
    H_m^{\mathrm{flat}}
      ={}&\frac{1}{(2\pi i)^{m-1}}
       \sum_{\mb n\in\N_0^m}
       \frac{(-1)^{|\mb n|+m-1}}{(\mb n!)^2}\\
      &\times\oint_{>1}\cdots\oint_{>1}
       \wt{\Sf D}^{(\mb n),\mathrm{flat}}_{\mb M,\mb N,\mb T}(\mb z)
       \prod_{i=1}^{m-1}
          \frac{(1+z_i)^{n_i-n_{i+1}-1}}{z_i^{n_{i+1}+1}}\,\dd z_i.
\end{split}
\end{equation}

We now recover the
\hyperref[not:terminal-sum]{terminal factor $S_{n_m}$}
appearing in \eqref{eq:Pi-main}. Only the level-$m$ exponential weights
depend on $T_m$. Since
$\partial_{T_m}\Sf f_m(w)=w\Sf f_m(w)$, differentiation gives
\begin{equation}
\label{eq:app-nc-derivative}
    \partial_{T_m}\Sf F^{(\mb n)}_{\mb M,\mb N,\mb T}
      =\left[\sum_{j=1}^{n_m}(\xi_j^m-\eta_j^m)\right]
          \Sf F^{(\mb n)}_{\mb M,\mb N,\mb T}
      =S_{n_m}(\bs\xi^m\mid\bs\eta^m)
          \Sf F^{(\mb n)}_{\mb M,\mb N,\mb T}.
\end{equation}
Thus the factor $S_{n_m}$ is inserted \emph{inside} the contour integral:
\begin{equation}
\label{eq:app-nc-D-derivative}
\begin{split}
    \partial_{T_m}
       \wt{\Sf D}^{(\mb n),\mathrm{flat}}_{\mb M,\mb N,\mb T}(\mb z)
       &=\mathcal I_{\mb n}(\mb z)
          \left[\Pi_{\mb n}\,
             \Sf F^{(\mb n)}_{\mb M,\mb N,\mb T}\,
             B_{n_1}^{\mathrm{flat}}\right]\\
       &={\Sf D}^{(\mb n),\mathrm{flat}}_{\mb M,\mb N,\mb T}(\mb z).
\end{split}
\end{equation}
This is the expression in \eqref{eq:Dnflat}.
The contours may be held fixed during differentiation. After expanding
the delta determinant and performing the corresponding evaluations,
the usual Hadamard bounds for the remaining Cauchy determinants give
locally uniform absolute convergence of the series and its $T_m$
derivative. Indeed, on these fixed contours the new factor $S_{n_m}$
is bounded by a constant times $n_m$, which is absorbed by the factorial
majorants. This justifies the termwise differentiation, as in the proof
of \cite[Proposition~3.1]{BCT}.

It follows that
\[
    \partial_{T_m}H_m^{\mathrm{flat}}
      =\sum_{\mb n\in\N_0^m}\frac{1}{(\mb n!)^2}
         \Sf Q_m^{(\mb n),\mathrm{flat}}(\mb M,\mb N,\mb T),
\]
where the summand is exactly \eqref{eq:Qnflat}. To recover the sum over
$\N^m$ in \eqref{eq:Qmflat-main}, it remains to remove the terms with zero
components. If $n_m=0$, then $S_{n_m}=0$. Otherwise, if some component
vanishes, there is an index $i<m$ such that $n_i=0$ and $n_{i+1}>0$.
The polynomial ${\Sf D}^{(\mb n),\mathrm{flat}}(\mb z)$ has degree at most
$2n_{i+1}$ in $z_i$, by \eqref{eq:app-nc-I}. For this index,
\[
    {\Sf D}^{(\mb n),\mathrm{flat}}(\mb z)
       \frac{(1+z_i)^{-n_{i+1}-1}}{z_i^{n_{i+1}+1}}
        =O(z_i^{-2})\qquad(z_i\to\infty).
\]
Its only possible finite poles in $z_i$ are at $0$ and $-1$, both
inside the contour. Hence its $z_i$ integral vanishes. This is the
argument of \cite[Lemma~3.2]{BCT}; it applies unchanged because the
flat factor is independent of $\mb z$. We have proved
\begin{equation}
\label{eq:app-nc-Q-density}
    \partial_{T_m}H_m^{\mathrm{flat}}
      =\sum_{\mb n\in\N^m}\frac{1}{(\mb n!)^2}
         \Sf Q_m^{(\mb n),\mathrm{flat}}(\mb M,\mb N,\mb T)
      =\Sf Q_m^{\mathrm{flat}}(\mb M,\mb N,\mb T).
\end{equation}

For $m=1$, the same definitions have no auxiliary variables or outer
contour integrals. They identify the denominator in
\eqref{eq:conditional-ratio-main} as the one-point density:
\begin{equation}
\label{eq:app-nc-one-point}
    \partial_T\p\big(\mathcal L^{\mathrm{flat}}(M,N)\leq T\big)
      =\Sf Q_1^{\mathrm{flat}}(M,N,T)
      =\sum_{n\geq1}\frac{(-1)^n}{(n!)^2}
          \Sf D_{M,N,T}^{(n),\mathrm{flat}}.
\end{equation}
Taking the ratio of the mixed density and this marginal density gives
\begin{equation}
\label{eq:app-nc-conditional}
\begin{split}
    &\p\big(E_{m-1}^{\mathrm{flat}}\mid
                    \mathcal L^{\mathrm{flat}}(M_m,N_m)=T_m\big)\\
    &\qquad=
       \frac{\partial_{T_m}H_m^{\mathrm{flat}}}
            {\partial_{T_m}\p\big(
                    \mathcal L^{\mathrm{flat}}(M_m,N_m)\leq T_m\big)}
       =\frac{\Sf Q_m^{\mathrm{flat}}(\mb M,\mb N,\mb T)}
              {\Sf Q_1^{\mathrm{flat}}(M_m,N_m,T_m)}.
\end{split}
\end{equation}
The ratio is understood wherever the marginal density is positive,
with conditioning interpreted as in \eqref{eq:flat-conditioning-convention}.
This is \eqref{eq:conditional-ratio-main} in Proposition~\ref{prop:multipointformula}.

\section{Comparison with the step formula}
\label{app:notation-step}

For the step initial condition $x_k(0)=-k$, $k\geq1$, the notation
conversion is identical except that the factor
$B_{n_1}^{\mathrm{flat}}$ is absent. Thus define
\begin{align}
\label{eq:app-nc-step-Dtilde}
    \wt{\Sf D}^{(\mb n)}_{\mb M,\mb N,\mb T}(\mb z)
       &:=\mathcal I_{\mb n}(\mb z)
          \left[\wt\Pi_{\mb n}\,
                \Sf F^{(\mb n)}_{\mb M,\mb N,\mb T}\right],\\
\label{eq:app-nc-step-D}
    {\Sf D}^{(\mb n)}_{\mb M,\mb N,\mb T}(\mb z)
       &:=\mathcal I_{\mb n}(\mb z)
          \left[\Pi_{\mb n}\,
                \Sf F^{(\mb n)}_{\mb M,\mb N,\mb T}\right]
        =\partial_{T_m}
          \wt{\Sf D}^{(\mb n)}_{\mb M,\mb N,\mb T}(\mb z).
\end{align}
The conversion from Liu's coefficients is
\begin{equation}
\label{eq:app-nc-step-conversion}
    \wt{\Sf D}^{(\mb n)}_{\mb M,\mb N,\mb T}(\mb z)
       =(-1)^{|\mb n|}\mathcal D_{\mb n,Y_{\mathrm{step}}}(-\mb z)
         \prod_{i=1}^{m-1}
            \frac{z_i^{n_{i+1}}}{(1+z_i)^{n_i-n_{i+1}}}.
\end{equation}
In particular, the direct conversion of Liu's coefficient gives the
\emph{undifferentiated} quantity $\wt{\Sf D}$; the quantity $\Sf D$
is obtained only after inserting $S_{n_m}$ through differentiation.

Define
\begin{align}
\label{eq:app-nc-step-Qn}
    \Sf Q_m^{(\mb n)}(\mb M,\mb N,\mb T)
      &:=\frac{(-1)^{|\mb n|+m-1}}{(2\pi i)^{m-1}}
       \oint_{>1}\cdots\oint_{>1}
          \Sf D_{\mb M,\mb N,\mb T}^{(\mb n)}(\mb z)
       \prod_{i=1}^{m-1}
          \frac{(1+z_i)^{n_i-n_{i+1}-1}}{z_i^{n_{i+1}+1}}\,\dd z_i,\\
\label{eq:app-nc-step-Q}
    \Sf Q_m(\mb M,\mb N,\mb T)
      &:=\sum_{\mb n\in\N^m}\frac{1}{(\mb n!)^2}
                 \Sf Q_m^{(\mb n)}(\mb M,\mb N,\mb T).
\end{align}
These are the quantities of \cite[Section~3]{BCT}. In that reference,
the associated field, denoted here by $\mathcal L^{\mathrm{step}}$,
is point-to-point LPP from $(1,1)$ with the starting-site weight
included, as in Definition~\ref{def:explpp}. Its TASEP correspondence is
\[
    \{\mathcal L^{\mathrm{step}}(M,N)>T\}
       =\{x_N(T)<M-N\},
\]
and the same density calculation yields
\[
    \p\left(
       \bigcap_{i=1}^{m-1}
          \{\mathcal L^{\mathrm{step}}(M_i,N_i)>T_i\}
       \,\middle|\,
          \mathcal L^{\mathrm{step}}(M_m,N_m)=T_m\right)
       =\frac{\Sf Q_m(\mb M,\mb N,\mb T)}
              {\Sf Q_1(M_m,N_m,T_m)}.
\]
Thus the flat and step formulas have the same exponential weights,
auxiliary-variable factors, differentiation rule, and restriction to
positive multi-indices. The flat formula replaces the initial Cauchy
factor $K_{n_1}(\bs\eta^1\mid\bs\xi^1)$ by the delta determinant,
as displayed in \eqref{eq:app-nc-flat-cancel}.

\bibliographystyle{plain}
\bibliography{refs}

@misc{borodin_fluctuations_2006,
	title = {Fluctuations in the discrete {TASEP} with periodic initial configurations and the {Airy}\_1 process},
	url = {http://arxiv.org/abs/math-ph/0611071},
	doi = {10.48550/arXiv.math-ph/0611071},
	urldate = {2026-08-06},
	publisher = {arXiv},
	author = {Borodin, Alexei and Ferrari, Patrik L. and Prähofer, Michael},
	month = nov,
	year = {2006},
	note = {arXiv:math-ph/0611071},
}

@article{borodin_large_2008,
	title = {Large time asymptotics of growth models on space-like paths {II}: {PNG} and parallel {TASEP}},
	volume = {283},
	issn = {0010-3616, 1432-0916},
	shorttitle = {Large time asymptotics of growth models on space-like paths {II}},
	url = {http://arxiv.org/abs/0707.4207},
	doi = {10.1007/s00220-008-0515-4},
	number = {2},
	urldate = {2026-08-06},
	journal = {Communications in Mathematical Physics},
	author = {Borodin, Alexei and Ferrari, Patrik L. and Sasamoto, Tomohiro},
	month = oct,
	year = {2008},
	note = {arXiv:0707.4207 [math-ph]},
	pages = {417--449},
}

@article{sasamoto_spatial_2005,
	title = {Spatial correlations of the {1D} {KPZ} surface on a flat substrate},
	volume = {38},
	issn = {0305-4470, 1361-6447},
	url = {http://arxiv.org/abs/cond-mat/0504417},
	doi = {10.1088/0305-4470/38/33/L01},
	number = {33},
	urldate = {2026-08-06},
	journal = {Journal of Physics A: Mathematical and General},
	author = {Sasamoto, T.},
	month = aug,
	year = {2005},
	note = {arXiv:cond-mat/0504417},
	pages = {L549--L556},
}

@article{borodin_fluctuation_2007,
	title = {Fluctuation properties of the {TASEP} with periodic initial configuration},
	volume = {129},
	issn = {0022-4715, 1572-9613},
	url = {http://arxiv.org/abs/math-ph/0608056},
	doi = {10.1007/s10955-007-9383-0},
	number = {5-6},
	urldate = {2026-08-06},
	journal = {Journal of Statistical Physics},
	author = {Borodin, Alexei and Ferrari, Patrik L. and Prähofer, Michael and Sasamoto, Tomohiro},
	month = dec,
	year = {2007},
	note = {arXiv:math-ph/0608056},
	pages = {1055--1080},
}

@article{matetski2021,
	title = {The {KPZ} fixed point},
	volume = {227},
	issn = {00015962, 18712509},
	url = {http://arxiv.org/abs/1701.00018},
	doi = {10.4310/ACTA.2021.v227.n1.a3},
	number = {1},
	urldate = {2026-07-06},
	journal = {Acta Mathematica},
	author = {Matetski, Konstantin and Quastel, Jeremy and Remenik, Daniel},
	year = {2021},
	note = {arXiv:1701.00018 [math.PR]},
	pages = {115--203},
}

@misc{BCT,
	title = {Conditional exponential directed last passage percolation under a one-point upper large deviation event},
	copyright = {Creative Commons Attribution 4.0 International},
	url = {https://arxiv.org/abs/2508.04954},
	doi = {10.48550/ARXIV.2508.04954},
	urldate = {2026-06-01},
	publisher = {arXiv},
	author = {Baik, Jinho and Cordaro, Dylan and Tripathi, Tejaswi},
	year = {2025},
	note = {Version Number: 2},
}

@misc{liu2022,
	title = {A conditional scaling limit of the {KPZ} fixed point with height tending to infinity at one location},
	url = {http://arxiv.org/abs/2208.12215},
	doi = {10.48550/arXiv.2208.12215},
	urldate = {2026-06-29},
	publisher = {arXiv},
	author = {Liu, Zhipeng and Wang, Yizao},
	month = oct,
	year = {2022},
	note = {arXiv:2208.12215 [math.PR]
version: 2},
}

@misc{liu2021,
	title = {Multi-point distribution of {TASEP}},
	url = {http://arxiv.org/abs/1907.09876},
	doi = {10.48550/arXiv.1907.09876},
	urldate = {2026-07-06},
	publisher = {arXiv},
	author = {Liu, Zhipeng},
	month = nov,
	year = {2021},
	note = {arXiv:1907.09876 [math.PR]},
}

@article{rost1981,
	title = {Non-equilibrium behaviour of a many particle process: {Density} profile and local equilibria},
	volume = {58},
	issn = {1432-2064},
	shorttitle = {Non-equilibrium behaviour of a many particle process},
	url = {https://doi.org/10.1007/BF00536194},
	doi = {10.1007/BF00536194},
	language = {en},
	number = {1},
	urldate = {2026-07-06},
	journal = {Zeitschrift für Wahrscheinlichkeitstheorie und Verwandte Gebiete},
	author = {Rost, H.},
	month = mar,
	year = {1981},
	pages = {41--53},
}

@misc{Joh00,
	title = {Shape {Fluctuations} and {Random} {Matrices}},
	url = {http://arxiv.org/abs/math/9903134},
	doi = {10.48550/arXiv.math/9903134},
	urldate = {2026-07-06},
	publisher = {arXiv},
	author = {Johansson, Kurt},
	month = oct,
	year = {1999},
	note = {arXiv:math/9903134},
}

@misc{baik_one-point_2026,
	title = {One-point fluctuations for exponential last passage percolation under upper-tail conditioning},
	url = {http://arxiv.org/abs/2608.20552},
	doi = {10.48550/arXiv.2608.20552},
	urldate = {2026-09-16},
	publisher = {arXiv},
	author = {Baik, Jinho and Tripathi, Tejaswi},
	month = aug,
	year = {2026},
	note = {arXiv:2608.20552 [math.PR]},
}

\end{document}